\documentclass[11pt,a4paper]{amsart}
\usepackage{mathrsfs}
\usepackage{amsfonts}
\usepackage[notref,notcite]{showkeys}
\usepackage{enumerate}
\usepackage[all]{xy}
\usepackage{amsmath,amssymb,xspace,amsthm}
\usepackage{tikz}
\usetikzlibrary{decorations.markings}

\newtheorem{theorem}{Theorem}
\newtheorem{example}[theorem]{Example}
\newtheorem{lemma}[theorem]{Lemma}
\newtheorem{proposition}[theorem]{Proposition}
\newtheorem{corollary}[theorem]{Corollary}
\numberwithin{equation}{section}

\newcommand{\C}{\ensuremath{\mathbb C}\xspace}

\renewcommand{\l}{\ensuremath{\lambda}}

\newcommand{\g}{\ensuremath{\mathfrak{g}}}

\newcommand{\h}{\ensuremath{\mathfrak{h}}}

\newcommand{\Z}{\ensuremath{\mathbb{Z}}\xspace}
\newcommand{\N}{\ensuremath{\mathbb{N}}\xspace}

\newcommand{\id}{\operatorname{Id}\xspace}

\renewcommand{\phi}{\varphi}

\renewcommand{\leq}{\leqslant}
\renewcommand{\geq}{\geqslant}

\def\mg{\mathfrak{g}}

\def\sl{\mathfrak{sl}}

\def\H{\mathcal{H}}

\def\s{\sigma}

\def\g{\mathfrak{g}}

\def\P{F}

\def\S{E}

\def\d{\delta}

\def\l{\mathfrak{l}}

\def\cP{\mathcal{P}}
\def\a{\mathbf{a}}

\allowdisplaybreaks

\begin{document}
\title
{Module structure on $U(\h)$ for Kac-Moody algebras}
\author{Yan-an Cai, Haijun Tan and Kaiming Zhao}
\date{}
\maketitle

\begin{abstract} Let $\g$ be an arbitrary Kac-Moody algebra with a Cartan sualgebra $\h$. In this paper, we determine the category of $\g$-modules that are  free $U(\h)$-modules of rank 1.
\end{abstract}


\vskip 5pt \noindent {\em 2000  Math. Subj. Class.:} 17B10, 17B20,
17B65, 17B66, 17B68

\vskip 10pt

\section{Introduction}
Lie algebra theory has become more and more widely used   in many areas of mathematics and physics. In order to study infinite-dimensional Lie algebras with a  root space decomposition similar to  finite dimensional simple Lie algebras,  Kac and Moody independently introduced Lie algebras associated with generalized Cartan matrices in late 1960s, now called Kac-Moody algebras. For the last few decades, these algebras (particularly affine Kac-Moody algebras) have played important roles in many other mathematical areas such as combinatorics, number theory, integrable systems, operator theory, quantum stochastic process and in quantum field theory of physics.

To understand the representation theory of an algebra, it is important to classify all its simple modules. This turns out to be very difficult for non-abelian Lie algebras, see \cite{B}. In fact, no complete classification of simple modules is known for any Kac-Moody algebras except for  the Lie algebra $\mathfrak{sl}_2$. Though the representation theory of Kac-Moody algebras is in general not well understood, some special classes of simple modules for certain families of Kac-Moody algebras are classified, in particular,  finite dimensional simple modules, simple highest weight modules for generic Kac-Moody algebras, simple weight modules with finite dimensional weight spaces for finite dimensional simple Lie algebras (see \cite{M}) and for affine Kac-Moody algebras (see \cite{DG}, \cite{FT}),  and simple Whittaker modules for  $\tilde{\mathfrak{sl}}_2$ (see \cite{ALZ}). There are other simple modules constructed over affine Kac-Moody algebras, see \cite{BBFK},  \cite{Chr},   \cite{FGM}, \cite{GZ},  \cite{LZ}, \cite{MZ}, \cite{R}, \cite{R1} and references therein.

Throughout this paper we denote by $ \Z, \Z_+, \N$, $\C$ and $\C^*$ the sets of all integers, nonnegative integers, positive integers, complex numbers and nonzero complex numbers, respectively. All algebras and vector spaces are over $\C$. Also, we denote by $U(\mathfrak{a})$ the enveloping algebra for any Lie algebra $\mathfrak{a}$ and denote by $\C[x_1,\cdots,x_n]$ the polynomial algebra in commuting variables $x_1,\cdots,x_n$ over $\C$.

 Let $\h$ be the standard Cartan subalgebra of a Kac-Moody algebra $\g$.
 Then a weight module over $\g$ is a $\g$-module on which $\h$ acts diagonalizably. The family of weight modules is one of the classical families of modules for Kac-Moody algebras. Let  $\H(\g)$ be the category of $\g$-modules defined by the ``opposite condition", namely, the full subcategory of $\g$-mod consisting of modules which are free of rank 1 when restricted to $U(\h)$.
In \cite{N}, Nilsson determined the category $\H(\mathfrak{sl}_{l+1})$. These modules are as follows.

\begin{example}
The Cartan matrix of $\mathfrak{sl}_{l+1}(\C)$ is
\[A=\begin{pmatrix}
2 & -1 & & &\\
-1 & 2 & -1 & &\\
 & \ddots & \ddots & \ddots & \\
 & & -1 & 2 & -1 \\
 & & & -1 & 2
\end{pmatrix}_{l\times l}.\]
Denote by $E_{i,j}(1\leq i,j\leq l+1)$ the $(l+1)\times(l+1)$ unit matrix with $1$ at the $(i,j)$-entry and  zero everywhere else. For $1\leq k\leq l+1$, let $$\bar{h}_{k}:=E_{k,k}-\frac{1}{l+1}\sum_{i=1}^{l+1}E_{i,i}.$$
Then $\bar{h}_{1}+\bar{h}_{2}+\cdots \bar{h}_{l+1}=0$, and \[  \{E_{i,j}| 1 \leq i\neq j \leq l+1\} \cup \{\bar{h}_{1}, \ldots \bar{h}_{l}\} \]
is a basis for $\mathfrak{sl}_{l+1}(\C)$.    For any $S\subseteq\{1,\cdots,l+1\}, b\in\C, a_{l+1}=1$, and $ \a=(a_1,\cdots, a_l)\in(\C^*)^l$,  the polynomial algebra $\C[\bar{h}_1,\cdots, \bar{h}_l]$ becomes an $\mathfrak{sl}_{l+1}$-module with the following action  for all $g\in\C[\bar{h}_1,\cdots, \bar{h}_l]$ and $i,j\in\{1,2,\cdots,   l+1\}, k\in\l=\{1,2,\cdots,l\}$:
\[
\begin{aligned}
h_k\cdot g&=h_kg, \\
E_{i,j}\cdot g&=a_ia_j^{-1}(\delta_{i\in S}+\delta_{i\not\in S}(h_i-b-1))(\delta_{j\in S}(h_j-b)+\delta_{j\not\in S})\bar{\sigma}_{i}\bar{\sigma}_j^{-1}(g),
\end{aligned}
\]
where $\bar{\sigma}_i$ $(1\leq i\leq l)$ is the algebra automorphism of $\C[\bar{h}_1,\cdots,\bar{h}_l]$ defined by mapping $\bar{h}_k$ to $\bar{h}_k-\delta_{k,i}$ while $\bar{\sigma}_{l+1}$ is the identity map, and
$$\delta_{s\in S}=\begin{cases}&1 \text{ if } s\in S\\ &0 \text{ if } s\notin S\end{cases}, \  \  \delta_{s\notin S}=\begin{cases}&0 \text{ if } s\in S\\ &1 \text{ if } s\notin S\end{cases}.
$$   We denote this module by $M(\a,b,S)$. See Lemma 5 in \cite{CLNZ}. Nilsson \cite{N} proved that \[
\H(\mathfrak{sl}_{l+1}(\C))=\{M(\a,b,S)|a\in(\C^*)^l,b\in\C,S\subseteq\{1,\cdots,l+1\}\}.
\]
Moreover, $M(\a,b,S)$ is simple if and only if $(l+1)b\not\in\Z$ or $S\neq\varnothing$ or $S\neq\{1,\cdots,l+1\}$.
\end{example}

For the convenience of our later use, we now rewrite the actions in the above example.
Let $\{e_i=E_{i,i+1},f_i=E_{i+1,i},h_i=E_{i,i}-E_{i+1,i+1}|i=1,\cdots,l\}$ be the set of Chevalley generators of  $\mathfrak{sl}_{l+1}(\C)$  and set
$$
(H_1,\cdots,H_l)^T=A^{-1}(h_1,\cdots,h_l)^T.$$
Then $ H_k=\sum\limits_{i=1}^k\bar{h}_i$  for all $k=1,\cdots,l.$
%
Replacing $a_ia_{i+1}^{-1}$ by $a_i$, the action of $\mathfrak{sl}_{l+1}(\C)$ on $M(\a,b,S)\cong\C[H_1,\cdots,H_l]$ becomes:
\[\begin{aligned}
&H_i\cdot g=H_ig;\\
&e_i\cdot g=\left\{\begin{array}{ll}
a_i(H_{i+1}-H_i-b)\s_i(g), i,i+1\in S,\\
a_i\s_i(g), i\in S,i+1\not\in S,\\
a_i(H_i-H_{i-1}-b-1)(H_{i+1}-H_i-b)\s_i(g), i\not\in S, i+1\in S,\\
a_i(H_i-H_{i-1}-b-1)\s_i(g), i,i+1\not\in S;
\end{array}\right.\\
&f_i\cdot g=\left\{\begin{array}{ll}
a_i^{-1}(H_i-H_{i-1}-b)\s_i^{-1}(g), i,i+1\in S,\\
a_i^{-1}(H_i-H_{i-1}-b)(H_{i+1}-H_i-b-1)\s_i^{-1}(g), i\in S,i+1\not\in S,\\
a_i^{-1}\s_i^{-1}(g), i\not\in S, i+1\in S,\\
a_i^{-1}(H_{i+1}-H_i-b-1)\s_i^{-1}(g), i,i+1\not\in S;
\end{array}\right.
\end{aligned}\]
where $g\in\C[H_1,\cdots, H_l]$, $i\in\l$, $ \s_i(g(H_1, H_2, \cdots, H_l))=g(H_1-\delta_{1, i}, H_2-\delta_{2, i}, \cdots, H_l-\delta_{l, i})$ and $H_{l+1}=H_0:=0$.


Recently, Nilsson proved that for a finite dimensional simple Kac-Moody algebra $\g$, $\H(\g)$ is empty except for $\g$ being of type $A_l$ and $C_l$, see \cite{N1}. Similar work for Witt algebras of all ranks was done in \cite{TZ}. Such modules for Heisenberg-Virasoro algebra and $W(2,2)$ were determined   in \cite{CG}. In \cite{CZ},   this kind of modules for basic Lie superalgebras  were determined. In the present paper, we focus on determining the module category $\H(\g)$ for all Kac-Moody algebras $\g$.

The paper is organized as follows. In section 2, we collect some definitions, notations and establish some preliminary  lemmas. The modules in $\H(\g)$  for affine Kac-Moody algebras $\g$ are determined at Theorem 3 in Section 3 while the modules in $\H(\g)$  for Kac-Moody algebras $\g$  of rank 2 are completely obtained at Theorem 11 in Section 4.  In Section 5 using our different approach  we recover Nilsson's result in \cite{N1} as a by-product, i.e., give all modules in $\H(\g)$  for Kac-Moody algebras of finite type.
Finally, in Section 6, we obtain all modules in $\H(\g)$  for general Kac-Moody algebras in Theorem 20.

\section{Preliminaries}

Let us first recall some notations from \cite{Ka}. By a {\emph {generalized Cartan matrix}} we mean a square integer matrix $A=(a_{ij})_{i,j\in\l}$ where $ \l=\{1,2,\cdots, l\}$ with the conditions
\begin{itemize}
\item $a_{ii}=2,i\in \l$;
\item $a_{ij}\leq0$ if $i\neq j$, and
\item $a_{ij}=0\Leftrightarrow a_{ji}=0$.
\end{itemize}
In this paper, we consider the case that $A$ is indecomposable, i.e. there is no partition $\l=\l_1\cup\l_2$  so that $a_{ij}=0$ whenever $i\in\l_1$ and $j\in\l_2$.
From \cite{Ka}, we know that a generalized Cartan matrix $A$ is one of the three types:  {\emph {finite}}, {\emph {affine}}, or {\emph {indefinite type}}.

The \emph{Dynkin diagram} $\Gamma(A)$ of a generalized Cartan matrix $A=(a_{ij})_{l\times l}$ is the graph with $l$ vertices numbered with $1, 2, \cdots,l$. If $i\neq j$ and $a_{ij}a_{ji}\leq4$ then we connect $i$ and $j$ with $\max(|a_{ij}|,|a_{ji}|)$ edges together with an arrow towards $i$ if $|a_{ij}|>1$. If $a_{ij}a_{ji}>4$ we draw a bold face line segment between $i$ and $j$ labeled with the ordered pair $(|a_{ij}|,|a_{ji}|)$.

Let $\langle\cdot,\cdot\rangle: \h\times\h^*\rightarrow\C$ denote the natural nondegenerate bilinear form pairing between a vector space $\h$ and its dual. Given:
\begin{itemize}
\item a generalized Cartan matrix $A=(a_{ij})_{i,j\in \l},$
\item a finite dimensional vector space $\h$ (Cartan subalgebra) with $\dim\h=2l-\mathrm{rank}A$, and
\item a choice of {\emph{simple roots}} $\Pi=\{\alpha_1,\cdots,\alpha_l\}\subseteq\h^*$ and \emph{simple coroots} $\Pi^\vee=\{\alpha_1^\vee,\cdots,\alpha_l^\vee\}\subseteq\h$ such that $\Pi$ and $\Pi^\vee$ are linearly independent with $\langle\alpha_j,\alpha_i^\vee\rangle=\alpha_j(\alpha_i^\vee)=a_{ij},i,j\in\l$,
\end{itemize}
we may associated to $A$ a Lie algebra $\g=\g(A)=\tilde{\g}(A)/\mathfrak{t}$ over a field $K$, called the \emph{Kac-Moody algebra}, where $\tilde{\g}(A)$ is the auxiliary Lie algebra generated by $\h$ and elements $\{e_i,f_i|i\in \l\}$ subject to relations:
\begin{eqnarray}\label{R2.1}\begin{aligned}
&[\h,\h]=0,\\
&[h,e_i]=\alpha_i(h)e_i, h\in\h,\\
&[h,f_i]=-\alpha_i(h)f_i, h\in\h,\\
&[e_i,f_j]=\delta_{ij}\alpha_i^\vee,\\
\end{aligned}\end{eqnarray}
and $\mathfrak{t}$ is the unique maximal ideal in $\tilde{\g}(A)$ that intersects $\h$ trivially (see Theorem 1.2 in \cite{Ka}). Gabber and Kac proved that for a symmetrizable generalized Cartan matrix $A$, in particular for $A$ of finite or affine type, $\g(A)$ is the Lie algebra generated by $\h$ and elements $\{e_i,f_i|i\in \l\}$ subject to relations (\ref{R2.1}) and (\cite{Ka}, Theorem 9.11):
\begin{eqnarray}\label{R2.2}
(\mathrm{ad}e_i)^{1-a_{ij}}(e_j)=(\mathrm{ad}f_i)^{1-a_{ij}}(f_j)=0, \forall i\neq j.
\end{eqnarray}

Let $\g$ be a Kac-Moody algebra over $\C$. The goal of this paper is to study the full subcategory $\H(\g)$ of $\g$-Mod consisting of objects whose restriction to $U(\h)$ are free of rank 1, i.e.
\[
\H(\g)=\{M\in \g\text{-Mod}|\mathrm{Res}_{\h}^{\g}M\cong_{\h}U(\h)\}.
\]
Take a basis  for $\h$ to be $\{H_1,H_2,\cdots,H_{2l-\mathrm{rank}A}\}$ which will be specified later. Since $\h$ is a commutative Lie algebra, we can identify $U(\h)$ with $\cP=\C[H_1,\cdots,H_{2l-\mathrm{rank}A}]$. For each $i\in J=\{1,2,\cdots,2l-\mathrm{rank}A\}$, we denote by
\[
\cP_i=\C[H_1,\cdots,H_{i-1},H_{i+1},\cdots,H_{2l-\mathrm{rank}A}].
\]
Clearly, $U(\h)=\cP_i[H_i],i\in J$. For $g\in U(\h)$, denote by $\deg_i(g)$ the \emph{degree of $g$} as a polynomial in indeterminant $H_i$ over $\cP_i$, and we set $\deg_i(0)=-1$.

Define the automorphisms of $U(\h)$ as follows:
 \[\begin{aligned}\sigma_i:&\ U(\h)\longrightarrow U(\h),\\
 &H_j\mapsto H_j-\delta_{i,j}, \forall j\in J.
 \end{aligned} \]
It is easy to see that $\s_i^{-1}$ is defined by mapping $H_j$ to $H_j+\delta_{i,j}$, and $\s_i\s_j=\s_j\s_i$ for all $i,j\in J$. It is easy to see the following fact.
\begin{lemma}
For $g\in\cP,k\in\Z\setminus\{0\}, m\in\Z_+$, we have
\[
\deg_i(\s_i^k-\mathrm{Id})(g)=\deg_ig-1, \text{ and}
\]

\[
\deg_i(\s_i-\id)^m(g)=\left\{\begin{array}{ll}
\deg_i g-m, &\text{if } m\leq\deg_i g,\\
-1, &\text{if } m>\deg_i g.
\end{array}\right.
\]
\end{lemma}

For a Kac-Moody algebra $\g$, the following auxiliary algebra will be useful in our later discussions. Let $Z$ be an abelian Lie algebra over $\C$. Then $\tilde{\g}=Z\oplus\g$ is a Lie algebra under the following Lie bracket:
\begin{equation}
[z,\g]=0, \forall z\in Z.
\end{equation}
Let $\tilde{\h}=Z\oplus\h$. Then $\tilde{\h}$ is an abelian Lie algebra and $U(\tilde{\h})
\cong U(Z)\otimes_{\C}U(\h)$. We can naturally extend the actions of $\sigma_i,i\in \l$ on $U(\h)$ to $U(\tilde{\h})$.

Assume that there is a $\tilde{\g}$-module structure on $U(\tilde{\h})$ which is a free $U(\tilde{\h})$-module of rank $1$. Since $\g$ is a subalgebra of $\tilde{\g}$, $U(\tilde{\h})$ is naturally  a $\g$-module. 
Denote by $\H(\tilde{\mg})$   the full subcategory of $U(\tilde{\mg})$-modules consisting of objects whose restriction to $U(\tilde{\h})$ are free of rank $1$.

\section{Modules for affine Kac-Moody algebras}

In this section we consider $\H(\g)$ for  affine Kac-Moody algebras $\g$. Indeed, we will prove   the following theorem.
\begin{theorem}\label{AFF}
Let $\g$ be an affine Kac-Moody algebra. Then
\[
\H(\g)=\varnothing.
\]
\end{theorem}

In \cite{Ka}, the author classified all generalized Cartan matrices of affine type:
\begin{itemize}
\item[(Aff1)] $A_1^{(1)}, A_l^{(1)}(l\geq2), B_l^{(1)}(l\geq3), C_l^{(1)}(l\geq2),D_l^{(1)}(l\geq4), E_l^{(1)}(l=6,7,8)$;
\item[(Aff2)] $A_2^{(2)},A_{2l}^{(2)}(l\geq2),A_{2l-1}^{(2)}(l\geq3),D_{l+1}^{(2)}(l\geq2),E_6^{(2)}$;
\item[(Aff3)] $D_4^{(3)}$.
\end{itemize}
From the realization of affine Kac-Moody algebras in \cite{Ka}, we  know that  any Kac-Moody algebra $\g$ of affine type 2 or 3  contains an affine type 1 subalgebra which has the same Cartan subalgebra with $\g$. Thus we only need to prove Theorem \ref{AFF} for   affine type 1 Kac-Moody algebras $\g$.

Let $\g$ be a Kac-Moody algebra of affine type 1. Then
\[
\g\cong\dot{\g}\otimes\C[t^{\pm1}]\oplus\C K\oplus\C d,
\]
with $\dot{\g}$ a finite-dimensional Kac-Moody algebra over $\C$ and
\[\begin{aligned}
&[x\otimes t^{k_1}+\lambda_1 K+\mu_1 d,y\otimes t^{k_2}+\lambda_2 K+\mu_2 d]\\
=&[x,y]\otimes t^{k_1+k_2}+\mu_1k_2y\otimes t^{k_2}-\mu_2k_1x\otimes t^{k_1}+k_1\delta_{k_1,-k_2}(x,y)K,
\end{aligned}\]
where $x,y\in\dot{\g},\lambda_i,\mu_i\in\C, k_i\in\Z$ and $(\cdot,\cdot)$ is an invariant bilinear form on $\dot{\g}$.

Taking a basis for $\dot\h$ (the Cartan subalgebra of $\dot\g$) to be $h_1,\cdots,h_l$ (Chevally generators), then
\[
\h=\C h_1+\cdots+\C h_l+\C K+\C d,
\]
and we have

\begin{lemma}
Let   $M\in\H(\g)$, and $(h_i\otimes t^k)\cdot1=R_{ik}\in M$. Then
\[
(h_i\otimes t^k)\cdot g=\tau^k(g)R_{ik}, \forall g\in M,
\]
where $\tau$ is the automorphism on $\C[h_1,\cdots,h_l,K,d]$ defined by mapping $d$ to $d-1$ with all other variables unchanged.
\end{lemma}

\begin{proof}
This follows from
\[
[h_i\otimes t^k, h_j]=[h_i\otimes t^k, K]=0, \forall i,j,k,
\]
and
\[
[d,h_i\otimes t^k]=kh_i\otimes t^k, \forall i,k.\qedhere
\]
\end{proof}

Now we can prove the following lemma,

\begin{lemma}
Let $\g$ be an affine type 1 Kac-Moody algebra. Then
\[
\H(\g)=\varnothing.
\]
\end{lemma}

\begin{proof}
Assume that $\H(\g)\neq\varnothing$ and let $M\in\H(\g)$. On $M\cong\C[h_1,\cdots,h_l,K,\\d]$, we have
\[\begin{aligned}
2kK&=2(kK)\cdot1\\
&=[h_1\otimes t^k,h_1\otimes t^{-k}]\cdot1\\
&=R_{1k}\tau^{k}(R_{1,-k})-R_{1,-k}\tau^{-k}(R_{1k})\\
&=\tau^{k}(R_{1,-k}\tau^{-k}(R_{1k}))-R_{1,-k}\tau^{-k}(R_{1k})\\
&=(\tau^{k}-1)(R_{1,-k}\tau^{-k}(R_{1k})),
\end{aligned}\]
where $i=1,2,\cdots, l$ and $k\in\Z$. By Lemma 2, we have
\[
R_{1,-k}\tau^{-k}(R_{1k})=-Kd+R_k,
\] with $\deg_dR_k=0.$
Therefore, either $\deg_dR_{1k}$ or $\deg_dR_{1,-k}$ is 1.

Take $k\neq j\in\Z$ such that $k+j\neq0$ and $\deg_dR_{1k}\deg_dR_{1j}=1$, say
\[
R_{1k}=a_kd+b_k, R_{1j}=a_jd+b_j,
\]
where $\deg_da_k=\deg_db_k=\deg_da_j=\deg_db_j=0$ and $a_ka_j\neq0$. Then
\[\begin{aligned}
0&=[h_1\otimes t^k, h_1\otimes t^j]\cdot1\\
&=R_{1k}\tau^k(R_{1j})-R_{1j}\tau^j(R_{1k})\\
&=a_ka_j(j-k)d+jb_ja_k-kb_ka_j
&\neq0,
\end{aligned}\]
which is impossible. Thus, $\H(\g)=\varnothing.$
\end{proof}

From the above lemma we know that Theorem \ref{AFF} holds.

\section{Modules for Kac-Moody algebras of rank 2}
In this section, we will study the category $\H(\g)$ for a Kac-Moody algebra $\g$ associated to generalized Cartan matrix $A(r,s)=(a_{ij})_{1\leq i,j\leq2}=\left(\begin{smallmatrix}2 & -r \\ -s & 2\end{smallmatrix}\right)$. 
Indeed, we will show that
\begin{proposition}
Let $\g$ be a Kac-Moody algebra associated to generalized Cartan matrix $A(r,s)$. Then $\H(\g)\neq\varnothing$ if and only if $rs\leq2$.
\end{proposition}
Since $A$ is always symmetrizable in this case, we may use relations (\ref{R2.2}) in our proofs. When $A$ is not invertible, then $rs=4$ and $\g$ is an affine Kac-Moody algebra, whose modules have been studied in Section 3. We may assume in our discussion that $rs\neq4$. For $r=s=1$, $\g$ is a finite dimensional simple Lie algebra of type $A_2$, the result for this algebra is a special case in Example 1. So, we may further assume that $rs\neq1$. 

Since $rs\neq4$, $A$ is invertible. We may take a basis for $\h$ to be
\[
(H_1,H_2)^{T}=A^{-1}(\alpha_1^\vee,\alpha_2^\vee)^{T}.
\]
Then we have
\[\begin{aligned}
&[H_i,e_j]=\d_{ij}e_j, [H_i,f_j]=-\d_{ij}f_j,\,\,i,j=1,2;\\
&[e_1,f_1]=2H_1-rH_2,\\
&[e_2,f_2]=-sH_1+2H_2,\\
&[e_i,f_j]=0, i\neq j,\\
&(\mathrm{ad}e_1)^{1+r}(e_2)=(\mathrm{ad}f_1)^{1+r}(f_2)=0,\\
&(\mathrm{ad}e_2)^{1+s}(e_1)=(\mathrm{ad}f_2)^{1+s}(f_1)=0.
\end{aligned}
\]
From the first two relations, we have
\begin{lemma}
Let  $M\in\H(\g)$. For any $g\in\C[H_1,H_2]\cong M$ and $ i\in\{1,2\}$, we have
\[
\begin{aligned}
&H_i\cdot g=H_ig,\\
&e_i\cdot g=\S_i\s_i(g),\\
&f_i\cdot g=\P_i\s_i^{-1}(g),
\end{aligned}
\]
where $\S_i=e_i\cdot 1,\P_i=f_i\cdot 1\in\C[H_1,H_2]\setminus\{0\}$.
\end{lemma}
Thus, to understand $M\in \H(\g)$, we only need to determine all possible $4$-tuples $(\S_1,\S_2,\P_1,\P_2)$.

Next, we will establish some useful lemmas for our later discussions.

\begin{lemma}\label{A_1}
Let $\g$ be the Lie algebra $\mathfrak{sl}_2(\C)$ and $Z$ be an abelian Lie algebra over $\C$. For any $a\in\C^*,b\in U(Z),S\subseteq\{1,2\}$,   the vector space $M(a,b,S)=U(\tilde{\h})$ becomes a $\tilde{\g}$-module with the following action for all $ \tilde{g}\in U(\tilde{\h})$:
\[\begin{aligned}
&H\cdot \tilde{g}=H\tilde{g}, z\cdot \tilde{g}=z\tilde{g},\forall z\in Z;\\
&e\cdot \tilde{g}=\left\{\begin{array}{ll}
a(H+b)\s_i(\tilde{g}), &1,2\not\in S\text{ or } 1,2\in S,\\
a\s_i(\tilde{g}), &1\in S,2\not\in S,\\
a(-H^2+H+b)\s_i(\tilde{g}), &1\not\in S, 2\in S;
\end{array}\right.\\
&f\cdot \tilde{g}=\left\{\begin{array}{ll}
a^{-1}(-H+b)\s_i^{-1}(\tilde{g}), &1,2\not\in S\text{ or } 1,2\in S,\\
-a^{-1}(H^2+H+b)\s_i^{-1}(\tilde{g}), &1\in S,2\not\in S,\\
a^{-1}\s_i^{-1}(\tilde{g}), &1\not\in S, 2\in S.
\end{array}\right.
\end{aligned}\]
Furthermore
\[\H(\tilde{\g})=\{M(a,b,S)|a\in\C^*,b\in U(Z),S\subseteq\{1,2\}\}.
\]
\end{lemma}

\begin{proof} Using the formulas after Example 1 for $n=1$ and re-choosing $a_i, b$, we see that the action of $\tilde{\g}$ on $U(\tilde{\h})$  defined above  gives a $\tilde{\g}$-module structure on $U(\tilde{\h})$, that is, $M(a,b, S)\in \H(\tilde{\g})$  for any   $\a\in\C^*,b\in U(Z),S\subseteq\{1,2\}$. So we only need to show that each module in  $\H(\tilde{\g})$ has the above defined structure.

 Let $M=U(\tilde{\h})$ be a module in $\H(\tilde{\g})$. By $[\mg,Z]=0$ and by the choice of the $H$ we see that  the $\tilde{\mg}$-module structure of $M$ is determined by the actions of the Chevalley generators  $e,f$ of $\mg$ on $1.$

Define a  new Lie algebra $\hat{\mg}:=U(Z)\otimes_{\C}\mg\subset U(\tilde{\g})$ over the PID $U(Z)$. Then $U(\tilde{\h})$becomes a $\hat{\mg}$-module over $U(Z)$ where the action of $\hat\mg$ on $U(\tilde{\h})$ inherits the action of  $\tilde{\mg}$  on  $U(\tilde{\h})$.

Now let us extend the  fundamental field. Let $\C(Z)$ be the fraction field of $U(Z)$ and let $\overline{\C(Z)}$ be the algebraically closed extension field of  $\C(Z)$. Let
$$G=\overline{\C(Z)}\otimes_{U(Z)}\hat{\mg}\cong \overline{\C(Z)}\otimes_{\C}\mg,
$$then $G\cong\sl_2(\overline{\C(Z)})$  with a Cartan subalgebra $\mathfrak{H}\cong \overline{\C(Z)}\otimes_{\C}\h$. If we identify $1\otimes \mg\subset G$ with $\mg$, then $(e,f,h)$ are the Chevalley generators for $G$.

Noting that $U(\tilde{\h})\cong U(Z)\otimes_{\C} U(\h)$, we have the vector space isomorphism over $\overline{\C(Z)}$ as follows:
$$\overline{\C(Z)}\otimes_{U(Z)}U(\tilde{\h})\cong \overline{\C(Z)}\otimes_{\C}U(\h)\cong \overline{\C(Z)}[\mathfrak{H}]=U(\mathfrak{H}),
$$
where the last term  is the universal enveloping algebra of $\mathfrak{H}$ over $\overline{\C(Z)}$.
So we can identify $U(\mathfrak{H})$ with $
\overline{\C(Z)}\otimes_{U(Z)}U(\tilde{\h})$. Clearly, it is natural to
extend the module action of $\hat{\mg}$ on $U(\tilde{\h})$ to the module action of $G$ on
$
\overline{\C(Z)}\otimes_{U(Z)}U(\tilde{ \h})$
by
$$\phi_1\otimes \phi_2\otimes x\circ \psi_1\otimes \psi_2 g=\phi_1\psi_1\otimes (\phi_2\otimes x \cdot \psi_2 g),
$$where $\phi_1,\psi_1\in \overline{\C(Z)}$, $\phi_2,\psi_2\in U(Z),$$x\in \mg, g\in U(\h)$.

 Since the action of $\mathfrak{H}$ on $
\overline{\C(Z)}\otimes_{U(Z)}U(\tilde{\h})$ is just multiplication, we see that the obtained $G$-module   $
\overline{\C(Z)}\otimes_{U(Z)}U(\tilde{\h})$ belongs to $\H(G)$, where all the modules are over $\overline{\C(Z)}$. Clearly, the method used to prove Theorems 11 and 30  in \cite{N} is valid for the $G$-module $\overline{\C(Z)}\otimes_{U(Z)}U(\tilde{\h})$, then the result in Example 1 holds for some $a\in \overline{\C(Z)}^*, b\in \overline{\C(Z)}, S\subseteq\{1,2\}$.
 Moreover, since the action of $\mg\subset G$ on $U(\tilde{\h})$ coincides with the action of  $\mg$ on the $\tilde{\mg}$-module $U(\tilde{\h})=M$,  we must have $e\circ 1, f\circ 1\in U(\tilde{\h})$, which forces that $a,b\in U(Z)$. Again by   Example 1 we see that $a^{-1}\in U(Z)$. So $a\in \C^*$ and $M\in\H(\g)$. Therefore, the lemma holds.
\end{proof}

Now consider $M$ as a module over the Lie algebras $$\aligned \C e_1+\C f_1+\C(2H_1-rH_2)+\C H_2, \\ \C e_2+\C f_2+\C(sH_1-2H_2)+\C H_1,\endaligned$$ separately. Using Lemma \ref{A_1}, we know that the only possible pairs $(\S_1,\P_1)$ and $(\S_2,\P_2)$ are
\[\begin{aligned}
(\S_1,\P_1)&=\left\{\begin{array}{l}
\left(a_1(H_1-\frac{r}{2}H_2+b),\,\,\, a_1^{-1}(-H_1+\frac{r}{2}H_2+b)\right),\\
\left(a_1,\,\,\, -a_1^{-1}((H_1-\frac{r}{2}H_2)^2+H_1-\frac{r}{2}H_2+b)\right),\\
\left(a_1(-(H_1-\frac{r}{2}H_2)^2+H_1-\frac{r}{2}H_2+b),\,\,\, a_1^{-1}\right);
\end{array}\right.\\
(\S_2,\P_2)&=\left\{\begin{array}{l}
\left(a_2(H_2-\frac{s}{2}H_1+c),\,\,\, a_2^{-1}(-H_2+\frac{s}{2}H_1+c)\right),\\
\left(a_2,\,\,\, -a_2^{-1}((H_2-\frac{s}{2}H_1)^2+H_2-\frac{s}{2}H_1+c)\right),\\
\left(a_2(-(H_2-\frac{s}{2}H_1)^2+H_2-\frac{s}{2}H_1+c),\,\,\, a_2^{-1}\right);
\end{array}\right.
\end{aligned}\]
where $b\in\C[H_2]$ and $c\in\C[H_1]$.

Before proving our main result in this section, we need to first establish some auxiliaries.

\begin{lemma}
Let $\g$ be the Kac-Moody algebra associated to the generalized  Cartan matrix $A(r,s)$ with $rs\neq1,4$. \begin{itemize}\item[(a).] If  $s>1$, then $\deg_1\S_1\deg_1\P_1=0$;\item[(b).]
If  $r>1$, then $\deg_2\S_2\deg_2\P_2=0$.\end{itemize}
\end{lemma}

\begin{proof} It suffices to show (a).
To the contrary, we assume that $$\S_1=a_1(H_1-\frac{r}{2}H_2+b),\,\,\,\, \P_1=a_1^{-1}(-H_1+\frac{r}{2}H_2+b)$$ with $b\in\C[H_2]$. If $\S_2=a_2$ and $\P_2=-a_2^{-1}((H_2-\frac{s}{2}H_1)^2+H_2-\frac{s}{2}H_1+c)$ for some $c\in\C[H_1]$, then
\[\begin{aligned}
0&=[e_2,f_1]\cdot1\\
&=\S_2\s_2(\P_1)-\P_1\s_1^{-1}(\S_2)\\
&=a_2(\s_2(\P_1)-\P_1).
\end{aligned}\]
This implies that $\deg_2\P_1=0$. Therefore $b=\frac{-r}{2}H_2+b_1$ with $b_1\in\C$. Thus,
\[
\S_1=a_1(H_1-rH_2+b_1),\P_1=a_1^{-1}(-H_1+b_1).
\]
Hence,
\[\begin{aligned}
0&=[e_1,f_2]\cdot1\\
&=\S_1\s_1(\P_2)-\P_2\s_2^{-1}(\S_1)\\
&=-a_1a_2^{-1}r(s-1)H_2^2+d_1H_2+d_0,
\end{aligned}\]
where $d_1,d_0$ are in $\C[H_1]$. So, this is a contradiction because $s>1$. Hence, $(\deg_2\S_2,\deg_2\P_2)\neq(0,2)$. Similarly, $(\deg_2\S_2,\deg_2\P_2)\neq(2,0)$.

Now let $\S_2=a_2(H_2-\frac{s}{2}H_1+c),\P_2=-a_2^{-1}(H_2-\frac{s}{2}H_1-c)$ with $c=\sum\limits_{k=0}^m\lambda_kH_1^k\in\C[H_1]$. If $m>1$ and $\lambda_m\ne0$, then
\[\begin{aligned}
0&=[e_1,f_2]\cdot1\\
&=\S_1\s_1(\P_2)-\P_2\s_2^{-1}(\S_1)\\
&=a_1a_2^{-1}\lambda_m(-m+\frac{r}{2}+b-\s_2^{-1}(b))H_1^m+\sum\limits_{k=0}^{m-1}\lambda'_kH_1^k,
\end{aligned}\]
where $\lambda'_k(k<m)$ are in $\C[H_2]$. Hence,
\[
b=-(m-\frac{r}{2})H_2+b_1, b_1\in\C.
\]
Thus,
\[\begin{aligned}
0&=[e_1,f_2]\cdot1&=a_1a_2^{-1}m(1-\frac{s}{2}+\s_1(c)-c)H_2+\lambda''_0,
\end{aligned}\]
where $\lambda''_0\in\C[H_1]$. Therefore, we have
\[
\s_1(c)-c=-(1-\frac{s}{2}).
\]
Hence, $m=\deg_1c\leq1$. This is a contradiction. Thus, $\deg_1c\leq1$. By symmetry we also have  $\deg_2b\leq1$.

Let $b=b_1H_2+b_2,c=c_1H_1+c_2$ with $b_1,b_2,c_1,c_2\in\C$. We see that
$$\aligned \S_1=a_1(H_1+(b_1-\frac{r}{2})H_2+b_2),\,\,\,\,\P_1=a_1^{-1}(-H_1+(b_1+\frac{r}{2})H_2+b_2),\\
\S_2=a_2(H_2+(c_1-\frac{s}{2})H_1+c_2),\,\,\,\,\P_2=-a_2^{-1}(H_2-(c_1+\frac{s}{2})H_1-c_2).\endaligned$$

Then
\[\begin{aligned}
0=&[e_1,f_2]\cdot1\\
=&a_2^{-1}a_1\left[(\frac{-s}{2}-c_1)(b_1+1-\frac{r}{2})H_1+(b_1-\frac{r}{2})(1-\frac{s}{2}-c_1)H_2\right.\\
&\left.-b_2(c_1+\frac{s}{2})-c_2(b_1-\frac{r}{2})\right].
\end{aligned}\]
Therefore, we have
\[
b_1=\frac{r}{2},c_1=\frac{-s}{2},
\]
or
\[
b_1=\frac{r}{2}-1,c_1=1-\frac{s}{2},b_2=c_2.
\]
If  $b_1=\frac{r}{2},c_1=-\frac{s}{2}$, then
\[\begin{aligned}
0&=[e_2,f_1]\cdot1\\
&=-s(1-r)H_1-r(1-s)H_2+c_2-r-b_2-s\\
&\neq0,
\end{aligned}\]
which is impossible.

If $b_1=\frac{r}{2}-1,c_1=1-\frac{s}{2},b_2=c_2$, then
\[\begin{aligned}
0&=[e_2,f_1]\cdot1\\
&=(2-s)(1-r)H_2+(2-r)(1-s)H_1+b_2(-r-s)\\
&\neq0.
\end{aligned}\]
Hence $s=r=2$, which contradicts our assumption that $sr\ne4$. This completes the proof.
\end{proof}

\begin{example}\label{C_2}
Let $\g$ be the Kac-Moody algebra of type $B_2$ with Cartan matrix $\left(\begin{smallmatrix}
2 & -1 \\
-2 & 2
\end{smallmatrix}\right)$. For any $\a=(a_1,a_2)\in(\C^*)^2,S\subseteq\{1,2\}$, the vector space
$M(\a,S)=\C[H_1,H_2]$ becomes a $\g$-module under the following action for  all $g\in\C[H_1,H_2]$:
\begin{equation}\label{R4.1}\begin{aligned}
&H_i\cdot g=H_ig,i=1,2,\\
&e_1\cdot g=\left\{\begin{array}{ll}
a_1\s_1(g), & 1\in S,\\
-a_1(H_1-\frac{1}{2}H_2-\frac{3}{4})(H_1-\frac{1}{2}H_2-\frac{1}{4})\s_1(g) & 1\not\in S;
\end{array}\right.\\
&f_1\cdot g=\left\{\begin{array}{ll}
-a_1^{-1}(H_1-\frac{1}{2}H_2+\frac{3}{4})(H_1-\frac{1}{2}H_2+\frac{1}{4})\s_1^{-1}(g), & 1\in S,\\
a_1^{-1}\s_1^{-1}(g), & 1\not\in S;
\end{array}\right.\\
&e_2\cdot g=\left\{\begin{array}{ll}
a_2(-2H_1+H_2-\frac{1}{2})\s_2(g), & 1,2\in S,\\
a_2(H_2-\frac{1}{2})(2H_1-H_2+\frac{1}{2})\s_2(g), &1\in S,2\not\in S,\\
a_2(H_2-\frac{1}{2})\s_2(g), & 1,2\not\in S,\\
a_2\s_2(g), & 1\not\in S,2\in S;
\end{array}\right.\end{aligned}\end{equation}
$$\begin{aligned}
&f_2\cdot g=\left\{\begin{array}{ll}
-a_2^{-1}(H_2+\frac{1}{2})\s_2^{-1}(g), & 1,2\in S,\\
a_2^{-1}\s_2^{-1}(g), &1\in S,2\not\in S,\\
a_2^{-1}(2H_1-H_2-\frac{1}{2})\s_2^{-1}(g), & 1,2\not\in S,\\
a_2^{-1}(H_2+\frac{1}{2})(2H_1-H_2-\frac{1}{2})\s_2^{-1}(g), & 1\not\in S,2\in S.
\end{array}\right.
\end{aligned}$$
This can be directly verified by using Serre's relations (Actually some will be verified in the proof of the next lemma). We omit the onerous  computations.
\end{example}

Now, let us prove our main result for Kac-Moody algebras of rank 2.

\begin{theorem}\label{RANK2}
Let $\g$ be the Kac-Moody algebra associated to the Cartan matrix $\left(\begin{smallmatrix}
2 & -r \\
-s & 2
\end{smallmatrix}\right)$. 
Then $\H(\g)\neq\varnothing$ if and only if $rs\leq2$. Moreover,
\begin{enumerate}[(1).]
\item If $r=s=1$, then
\[
\H(\g)=\{M(\a,b,S)|\a\in(\C^*)^2,b\in\C,S\subseteq\{1,2,3\}\},
\]
where $M(\a,b,S)$ are as defined in Example 1.
\item If  $r=1,s=2$, then
\[
\H(\g)=\{M(\a,S)|\a\in(\C^*)^2,S\subseteq\{1,2\}\},
\]
where $M(\a,S)$ are as defined in Example \ref{C_2}.
\end{enumerate}
\end{theorem}

\begin{proof}

From Example 1 and Theorem \ref{AFF}, we may assume that $rs\neq4$ or $1$. Then we may assume that  $s\geq2$. Suppose that $\S_1,\P_1,\S_2,\P_2$ determine a module $M$ in $\H(\g)$. By using Chevalley involution and Lemma \ref{A_1}, we can also assume that
\[
\S_1=a_1,\P_1=-a_1^{-1}[(H_1-\frac{r}{2}H_2)^2+H_1-\frac{r}{2}H_2+b],
\]
with $b=\sum\limits_{j=0}^k\lambda_jH_2^j\in\C[H_2]$ and $a_1\in\C^*$.

Therefore, we have $\deg_2\P_1\neq0$, which implies that $\deg_1\S_2\neq0$. Indeed, if $\deg_1\S_2=0$, then $\S_2\in\C$ and from
\[
0=[e_2,f_1]\cdot1=\S_2(\s_2(\P_1)-\P_1),
\]
we get $\deg_2\P_1=0$, a contradiction. Since
\[
0=[e_1,f_2]\cdot1=a_1(\s_1(\P_2)-\P_2),
\]
we have $\deg_1\P_2=0$. So
\[
(\S_2,\P_2)=(a_2(H_2-sH_1+c_1),a_2^{-1}(-H_2+c_1)),
\]
with $c_1\in\C$, or
\[
(\S_2,\P_2)=(a_2(-(H_2-\frac{s}{2}H_1)^2+H_2-\frac{s}{2}H_1+c),a_2^{-1})
\]
with $c\in\C[H_1]$.

\

{\bf Case i:} $(\S_2,\P_2)=\left(a_2(H_2-sH_1+c_1),a_2^{-1}(-H_2+c_1)\right),c_1\in\C$.

In this case, following from Lemma 9, we have $r=1$. Therefore,
\[\begin{aligned}
0=&(\mathrm{ad}f_1)^2(f_2)\cdot1\\
=&\P_2\big(\P_1\s_1^{-1}(\P_1)-2\P_1\s_1^{-1}\s_2^{-1}(\P_1)+\s_2^{-1}(\P_1)\s_1^{-1}\s_2^{-1}(\P_1)\big)\\
=&a_1^{-2}\P_2\big((b-\s_2^{-1}(b))^2+\frac{1}{2}(b+\s_2^{-1}(b))-\frac{3}{16}\big).
\end{aligned}\]
Hence,
\begin{equation}
(b-\s_2^{-1}(b))^2+\frac{1}{2}(b+\s_2^{-1}(b))-\frac{3}{16}=0.
\end{equation}
If $k>2$, then the leading term of (4.2) is $\lambda_k^2k^2H_2^{2k-2}$, which is not $0$. If $k=1$, the leading term of (4.2) is $\lambda_1H_2$, which is absurd.
So, $k=0$ or $2$.

Suppose $k\leq2$ and $k\neq1$. Then $\P_1=-a_1^{-1}\big((H_1-\frac{1}{2}H_2)^2+H_1-\frac{1}{2}H_2+\lambda_2H_2^2+\lambda_1H_2+\lambda_0\big)$, and therefore
\[\begin{aligned}
0=&[e_2,f_1]\cdot1\\
=&\S_2(\s_2(\P_1)-\P_1)+\P_1(\S_2-\s_1^{-1}(\S_2))\\
=&-a_2a_1^{-1}\big[(H_2-sH_1+c_1)(H_1-(\frac{1}{2}+2\lambda_2)H_2+\frac{3}{4}+\lambda_2-\lambda_1)\\
&+s\big((H_1-\frac{1}{2}H_2)^2+H_1-\frac{1}{2}H_2+\lambda_2H_2^2+\lambda_1H_2+\lambda_0\big)\big]\\
=&-a_2a_1^{-1}(\lambda_2+\frac{1}{4})(s-2)H_2^2-\frac{1}{2}a_2a_1^{-1}(2+4s\lambda_2-s)H_1H_2\\
&+(\lambda_1s+c_1-\lambda_2s+\frac{s}{4})H_1+(\lambda_1s-2\lambda_2c_1+\lambda_2-\frac{c_1}{2}+\frac{3}{4}-\lambda_1-\frac{s}{2})H_2\\
&+c_1(\lambda_2-\lambda_1+\frac{3}{4})+s\lambda_0.
\end{aligned}\]
 So
\[\begin{aligned}
&(\lambda_2+\frac{1}{4})(s-2)=0,\\
&2+4s\lambda_2-s=0,\\
&\lambda_1s+c_1-\lambda_2s+\frac{s}{4}=0,\\
&\lambda_1s-2\lambda_2c_1+\lambda_2-\frac{c_1}{2}+\frac{3}{4}-\lambda_1-\frac{s}{2}=0,\\
&c_1(\lambda_2-\lambda_1+\frac{3}{4})+s\lambda_0=0.
\end{aligned}\]
Hence, from the first two equations and $s\geq2$, we have $s=2$ and $\lambda_2=\frac{1}{4}-\frac{1}{2s}=0$. Therefore, $\lambda_1=0$ and
\[
s=2,c_1=-\frac{1}{2}, b=\frac{3}{16},
\]
i.e.,
\[\begin{aligned}
&\P_1=-a_1^{-1}(H_1-\frac{1}{2}H_2+\frac{3}{4})(H_1-\frac{1}{2}H_2+\frac{1}{4}),\\
&\S_2=a_2(H_2-2H_1-\frac{1}{2}).
\end{aligned}\]
Thus, $r=1,s=2$ and $M\cong M((a_1,a_2),\{1,2\})$.

\

{\bf Case ii:} $(\S_2,\P_2)=(a_2(-(H_2-\frac{s}{2}H_1)^2+H_2-\frac{s}{2}H_1+c),a_2^{-1})$ with $c=\sum\limits_{j=0}^{k'}\mu_jH_1^j\in\C[H_1]$.

\

From $[e_2,f_1]=0$ we see that
\begin{equation}\aligned & \left(r(H_1-\frac{r}{2}H_2)+\frac{r^2}{4}+\frac{r}{2}+\sigma_2b-b\right)\left( -(H_2-\frac{s}{2}H_1)^2+H_2-\frac{s}{2}H_1+c\right)\\
&=\left((H_1-\frac{r}{2}H_2)^2+H_1-\frac{r}{2}H_2+b\right)\left( s(H_2-\frac{s}{2}H_1)-\frac{s^2}{4}-\frac{s}{2}+\sigma_1^{-1}c-c\right).\endaligned
\end{equation}

If $k>2$, from (4.3) we deduce that
$$
(s-k)H_2^{k+1}+\sum\limits_{i=0}^{k}d_iH_2^i=0,
$$
where $d_i\in\C[H_1]$. So, $k=s$. From (4.3) again we see that
\[\aligned&
\big(\lambda_s(-\frac{s^2}{2}H_1-\frac{s^2}{4}+c-\s_1^{-1}(c))
-\lambda_{s-1}-\delta_{s,3}(s-2)\frac{r^2}{4}\big)H_2^{s}\\
&+(\text{lower terms in } H_2)=0.\endaligned
\]
So,
\[
c-\s_1^{-1}(c)=\frac{s^2}{2}H_1+\frac{s^2}{4}+\lambda_s^{-1}(\lambda_{s-1}+\delta_{s,3}(s-2)\frac{r^2}{4}),
\]
which implies that
\[
c=-\frac{s^2}{4}H_1^2-(\lambda_s^{-1}(\lambda_{s-1}+\delta_{s,3}(s-2)\frac{r^2}{4})H_1+\mu_0, \mu_0\in\C.
\]
In this case, expressing (4.3) as a polynomial in $H_1$ we obtain that
\[
s^2(2-r)H_1^3+(\text{lower terms in } H_1)=0.
\]
Therefore, $r=2$. Comparing the coefficients of $H_1^2$ in (4.3) we see that
$$\aligned &\frac{s^2}2(H_2-2+b-\sigma_2(b))+2(sH_2-s/2-\lambda_s^{-1}(\lambda_{s-1}+\delta_{s,3}))\\
&=sH_2- \frac{s^2}2-s/2-\lambda_s^{-1}(\lambda_{s-1}+\delta_{s,3})+(2H_2-1)s^2,\endaligned$$
yielding that $k=s=2$, contradicting the assumption that $rs\ne4$. So we have $k\le2$. (Note that we have not used the assumption $s\ge 2$ in proving $k\le2$).

Using the notation change: $e_1\leftrightarrow f_2, e_2 \leftrightarrow f_1,
h_1\leftrightarrow -h_2, r\leftrightarrow s$ we also have $k'\le2$.


Now   the coefficients of $H_1^3$ and $H_2^3$ in (4.3) are
$$\frac14(4\mu_2-s^2)(r-2), -\frac14(s-2)(4\lambda_2+r^2),$$
which should be $0$. Since both $r$ and $s$ cannot be $2$, we see that $r\ne2$ or $s\ne2$.

If $s\ne2$, we have $\lambda_2=-\frac{r^2}4$. Now the coefficient of $H_2^2H_1$   in (4.3) is $r(s-1)$, yielding that $s=1$, a contradiction.
Hence $r\ne2$, and further  $\mu_2=\frac{s^2}4$.  Now the coefficient of $H_2H_1^2$   in (4.3) is $s(r-1)$, yielding that $r=1$.  If $\lambda_2=-\frac14$, the coefficients of $H_2^2H_1$   in (4.3) is $s-1$, yielding that $s=1$, contradicting to the assumption that $rs\ne1$. Thus $\lambda_2\ne-\frac14$. So $s=2$.
We summarize what I got: $r=1, s=2, \mu_2=1$.

Computing the coefficients of $H_2^2H_1$ and $H_1H_2$  in (4.3) we know that $\lambda_2=0$ and $\mu_1=4\lambda_1$. Now (4.3) becomes
$$\aligned &(1/4-4\lambda_1^2+\mu_0)H_1-2\lambda_1H_2^2+(2\lambda_1-(1/2)\mu_0-2\lambda_0\\
&+1/4-4\lambda_1^2)H_2+\lambda_0+(3/4)\mu_0-\lambda_1\mu_0-4\lambda_0\lambda_1=0.\endaligned$$
yielding $\mu_0=-1/4, \lambda_0=3/16.$  Therefore $$ r=1, s=2, \mu_2=1, \mu_0=-1/4, \lambda_0=3/16, \mu_1=\lambda_1=\lambda_2=0.$$

Moreover   \[\aligned & b=\frac{3}{16}, \,\,\,
c=H_1^2-\frac{1}{4},\\
&\P_1=-a_1^{-1}(H_1-\frac{1}{2}H_2+\frac{3}{4})(H_1-\frac{1}{2}H_2+\frac{1}{4}),\\
&\S_2=a_2(H_2-\frac{1}{2})(2H_1-H_2+\frac{1}{2}).\endaligned
\]
Hence, $M\cong M(\{a_1,a_2\},\{1\})$.
\end{proof}

\section{Modules for Kac-Moody algebras of finite type}

In this section, we will recover Nilsson's results in \cite{N1} with a totally different approach, i.e., determine $\H(\g)$ for finite Kac-Moody algebras $\g$.

Since  $A$ is of finite type , $A$ is invertible. We may take a basis for $\h$ to be
\[
(H_1,\cdots,H_l)^{T}=A^{-1}(\alpha_1^\vee,\cdots,\alpha_l^\vee)^{T}.
\]
Then we have
\[\begin{aligned}
&[H_i,e_j]=\d_{ij}e_j, [H_i,f_j]=-\d_{ij}f_j,\,\,i,j=1,\cdots,l;\\
&([e_1,f_1], [e_2,f_2], \cdots, [e_l,f_l])^{T}=(\alpha_1^\vee,\cdots,\alpha_l^\vee)^{T}=A(H_1,H_2,\cdots,H_l)^{T},\\
&[e_i,f_j]=0, i\neq j,\,\,i,j=1,\cdots,l.
\end{aligned}
\]
It is well-known that generalized Cartan matrices of finite type are given as follows (\cite{Ka}, \cite{H}):
\[
A_l,B_l\,\, (l\geq2),C_l\,\, (l\geq2),D_l\,\, (l\geq4),E_l\,\, (l=6,7,8), F_4,G_2.
\]

The following result is clear.

\begin{lemma}
Let $\g$ be a simple Lie algebra of finite type $X_l$ and $M\in\H(\g)$. For any $g\in M=\C[H_1,\cdots,H_l]$ and $ i\in\l$, we have
\[
\begin{aligned}
&H_i\cdot g=H_ig,\\
&e_i\cdot g=\S_i\s_i(g),\\
&f_i\cdot g=\P_i\s_i^{-1}(g),
\end{aligned}
\]
where $\S_i=e_i\cdot 1,\P_i=f_i\cdot 1\in\C[H_1,\cdots,H_l]\setminus\{0\}$.
\end{lemma}
Thus, to understand $\H(\g)$, we only need to determine all possible $2l$-tuples $(\S_1,\cdots,\S_l,\P_1,\cdots,\P_l)$ of $\C[H_1,\cdots,H_l]$.

\




Similar as Lemma \ref{A_1}, we can prove

\begin{lemma}\label{AXLEM}
Let $\g$ be a Kac-Moody algebra and  $Z$ be an abelian Lie algebra over $\C$.
\begin{enumerate}[(1).]
\item Let $\g$ be of type $A_l$ with $l\geq2$. For any $\a=(a_1,\cdots, a_l)\in(\C^*)^l, b\in U(Z)$ and $S\subseteq\l\cup\{l+1\}$, the space
 $M(\a,b,S)=U(\tilde{\h})$ becomes a   $\tilde{\g}$-module with the following action for all $\tilde{g}\in U(\tilde{\h})$:
\[\begin{aligned}
&H_i\cdot \tilde{g}=H_i\tilde{g}, z\cdot \tilde{g}=z\tilde{g}, i\in\l, \forall z\in Z;\\
&e_i\cdot \tilde{g}=\left\{\begin{array}{ll}
a_i(H_{i+1}-H_i-b)\s_i(\tilde{g}), &i,i+1\in S,\\
a_i\s_i(\tilde{g}), &i\in S,i+1\not\in S,\\
a_i(H_i-H_{i-1}-b-1)(H_{i+1}-H_i-b)\s_i(\tilde{g}), &i\not\in S, i+1\in S,\\
a_i(H_i-H_{i-1}-b-1)\s_i(\tilde{g}), &i,i+1\not\in S;
\end{array}\right.\\
&f_i\cdot \tilde{g}=\left\{\begin{array}{ll}
a_i^{-1}(H_i-H_{i-1}-b)\s_i^{-1}(\tilde{g}), &i,i+1\in S,\\
a_i^{-1}(H_i-H_{i-1}-b)(H_{i+1}-H_i-b-1)\s_i^{-1}(\tilde{g}), &i\in S,i+1\not\in S,\\
a_i^{-1}\s_i^{-1}(\tilde{g}), &i\not\in S, i+1\in S,\\
a_i^{-1}(H_{i+1}-H_i-b-1)\s_i^{-1}(\tilde{g}), &i,i+1\not\in S;
\end{array}\right.
\end{aligned}\]
where $H_{l+1}=H_0:=0$. And we have
\[
\H(\tilde{\g})=\{M(\a,b,S)|\a\in(\C^*)^l,b\in U(Z),S\subseteq\l\cup\{l+1\}\}.
\]
\item Let $\g$ be of type $B_2$. Let $\a=(a_1,a_2)\in(\C^*)^2$ and let $\S_{i,S}=e_i\cdot1,\P_{i,S}=f_i\cdot1\,\,(i=1,2,S\subseteq\{1,2\})$ be as in equation (4.1).   The space  $M(\a,S)=U(\tilde{\h})$ becomes a   $\tilde{\g}$-module with the following action for all $\g\in U(\tilde{\h}), z\in Z$ and $i=1,2$:\begin{equation}\label{R5.1}\begin{aligned}
&H_i\cdot g=H_ig, \\ &z\cdot g=zg,\\
&e_i\cdot g=\S_{i,S}\s_i(g),\\
&f_i\cdot g=\P_{i,S}\s_i^{-1}(g).
\end{aligned}\end{equation}
And we have $\H(\tilde{\g})=\{M(\a,S)|\a\in(\C^*)^2,S\subseteq\{1,2\}\}$.
\item $\H(\g)=\varnothing$ if and only if  $\H(\tilde{\g})=\varnothing$.
\end{enumerate}
\end{lemma}

Using Lemma \ref{AXLEM}, we can obtain the following

\begin{proposition}
If $\g$ is of type $D_4$, then $\H(\g)=\varnothing$.
\end{proposition}

\begin{proof}
Suppose the Dynkin diagram of $\g$ is as follow
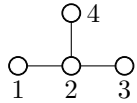
\[
\begin{tikzpicture}[scale=0.7]
    \draw (2,0) -- (4,0);
    \draw (3,0) -- (3,1);
    \filldraw[color=black] (2,0) circle (5pt) node[below=2pt]{\small 1};
    \filldraw[color=white] (2,0) circle (4pt);
    \filldraw[color=black] (3,0) circle (5pt) node[below=2pt]{\small 2};
    \filldraw[color=white] (3,0) circle (4pt);
    \filldraw[color=black] (4,0) circle (5pt) node[below=2pt]{\small 3};
    \filldraw[color=white] (4,0) circle (4pt);
    \filldraw[color=black] (3,1) circle (5pt) node[right=2pt]{\small 4};
    \filldraw[color=white] (3,1) circle (4pt);
\end{tikzpicture}
\]
To the contrary, assume that $\H(\g)\neq\varnothing$ and take $M\in\H(\g)$. For $k=1,3,4$, let $\g_k$ be the subalgebra of $\g$ generated by $\{e_i,f_i,\h|i\in\{1,2,3,4\},i\neq k\}$. Then
\[
\g_k\cong\mathfrak{sl}_4(\C)\oplus\C H_k,
\]
where $[\sl_4(\C),H_k]=0$. Let $M_k=\mathrm{Res}_{\g_k}^{\g}M$ for $k=1,3,4$. Take $\h_k$ to be the subalgebra of $\h$ generated by $\{\alpha_i^\vee|i\in\{1,2,3,4\}\setminus\{k\}\}$.

Let $\bar{H}_i\in\h_4$ with the property
\[
[\bar{H}_i,e_j]=\delta_{ij}e_j,[\bar{H}_i,f_j]=-\delta_{ij}f_j,i,j=1,2,3.
\]
From $$\aligned
2\bar H_1-\bar H_2&=\alpha_1^\vee=2 H_1- H_2,\\
-\bar H_1+2\bar H_2-\bar H_3&=\alpha_2^\vee=-H_1+2H_2-H_3-H_4,\\
-\bar H_2+2\bar H_3&=\alpha_3^\vee=-H_2+2H_3,
\endaligned$$ we see that
\[
\bar{H}_3=H_3-\frac{1}{2}H_4, \bar{H}_i=H_i-\frac{i}{2}H_4,i=1,2.
\]
Considering $M_4$, by Lemma 17, we have $(\S_1,\P_1)$ is one of the following
\[\left\{\begin{array}{l}
(a_1(H_2-H_1-\frac{1}{2}H_4-b),a_1^{-1}(H_1-\frac{1}{2}H_4-b)),\\
(a_1,a_1^{-1}(H_1-\frac{1}{2}H_4-b)(H_2-H_1-\frac{1}{2}H_4-b-1)),\\
(a_1(H_1-\frac{1}{2}H_4-b-1)(H_2-H_1-\frac{1}{2}H_4-b),a_1^{-1}),\\
(a_1(H_1-\frac{1}{2}H_4-b-1),a_1^{-1}(H_2-H_1-\frac{1}{2}H_4-b-1));
\end{array}\right.\]
and $(\S_3,\P_3)$ is one of the following respectively
\[\left\{\begin{array}{l}
(-a_3(H_3-\frac{1}{2}H_4+b),a_3^{-1}(H_3-H_2+\frac{1}{2}H_4-b)),\\
(-a_3(H_3-H_2+\frac{1}{2}H_4-b-1)(H_3-\frac{1}{2}H_4+b),a_3^{-1}),\\
(a_3,-a_3^{-1}(H_3-H_2+\frac{1}{2}H_4-b)(H_3-\frac{1}{2}H_4+b+1)),\\
(a_3(H_3-H_2+\frac{1}{2}H_4-b-1),-a_3^{-1}(H_3-\frac{1}{2}H_4+b+1));
\end{array}\right.\]
where $a_1,a_3\in\C^*$ and $b\in\C[H_4]$. Therefore, we have
\[
\deg_3\S_1=\deg_3\P_1=\deg_1\S_3=\deg_1\P_3=0.
\]
Similarly, considering $M_1$, we have
\[
\deg_4\S_3=\deg_4\P_3=0.
\]
This is impossible. Therefore, $M$ does not exist, namely, $\H(\g)=\varnothing$.
\end{proof}

If $\g$ is a Kac-Moody algebra whose Dynkin diagram $\Gamma$ contains a subdiagram of type $D_4$, that is $\Gamma$ is of the form:
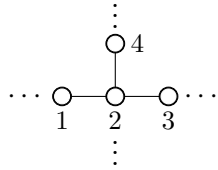
\[
\begin{tikzpicture}[scale=0.7]
    \draw (2,0) -- (4,0);
    \draw (3,0) -- (3,1);
    \filldraw (1.9,0) node[left]{$\cdots$};
    \filldraw[color=black] (2,0) circle (5pt) node[below=2pt]{\small 1};
    \filldraw[color=white] (2,0) circle (4pt);
    \filldraw[color=black] (3,0) circle (5pt) node[below=2pt]{\small 2};
    \filldraw[color=white] (3,0) circle (4pt);
    \filldraw[color=black] (4,0) circle (5pt) node[below=2pt]{\small 3};
    \filldraw[color=white] (4,0) circle (4pt);
    \filldraw (4.1,0) node[right]{$\cdots$};
    \filldraw[color=black] (3,1) circle (5pt) node[right=2pt]{\small 4};
    \filldraw[color=white] (3,1) circle (4pt);
    \filldraw (3,1.1) node[above]{$\vdots$};
    \filldraw (3,-0.1) node[below=5pt]{$\vdots$};
\end{tikzpicture}
\]
By a subdiagram of $\Gamma$ we mean a subset of vertices of $\Gamma$ and all edges between them.
Then $\{e_i,f_i,\h|i=1,2,3,4\}$ generates a subalgebra $\g_1$ of $\g$ that is isomorphic to $\mathfrak{o}_8(\C)\oplus Z$ for some abelian Lie algebra $Z$. Therefore, for any $M\in\H(\g)$,
\[\mathrm{Res}_{\g_1}^{\g}M\in\H(\g_1)=\varnothing.\]
Thus, we have

\begin{lemma}\label{D_4}
Let $\g$ be a Kac-Moody algebra with Dynkin diagram $\Gamma$. If $\Gamma$ contains a subdiagram of type $D_4$, then $\H(\g)=\varnothing$.
\end{lemma}

Following from this lemma, we can easily get

\begin{corollary}\label{DE}
Let $\g$ be a finite Kac-Moody algebra of type $D_l (l\geq4)$ or $E_l (l=6,7,8)$. Then $\H(\g)=\varnothing$.
\end{corollary}

Next, we consider Kac-Moody algebras $\g$ whose Dynkin diagram  contains a subdiagram of type $B_3$.

\begin{lemma}\label{B,F}
If $\g$ is a Kac-Moody algebra of type $B_l (l\geq3)$ or $F_4$, then $\H(\g)=\varnothing$.
\end{lemma}

\begin{proof} Since the Dynkin diagram  of $\g$ contains a subdiagram of type $B_3$,
using the same argument as Lemma \ref{D_4}, we only need to prove the statement for $\g$ of type $B_3$.

Let the Dynkin diagram of $\g$ be as follows.
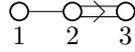
\[
\begin{tikzpicture}[scale=0.7]
    \draw (1,0) -- (2,0);
    \draw (2,0.1) -- (3,0.1);
    \draw (2,-0.1) -- (3,-0.1);
    \draw (2.4,0.2) -- (2.6,0) -- (2.4,-0.2);
    \filldraw[color=black] (1,0) circle (5pt) node[below=2pt]{\small 1};
    \filldraw[color=white] (1,0) circle (4pt);
    \filldraw[color=black] (2,0) circle (5pt) node[below=2pt]{\small 2};
    \filldraw[color=white] (2,0) circle (4pt);
    \filldraw[color=black] (3,0) circle (5pt) node[below=2pt]{\small 3};
    \filldraw[color=white] (3,0) circle (4pt);
\end{tikzpicture}
\]
Suppose $\H(\g)\neq\varnothing$ and take $M\in\H(\g)$. For $k=1,3$ let $\g_k$ be the subalgebra of $\g$ generated by $\{e_i,f_i,\h|i\in\{1,2,3\}\setminus\{k\}\}$. Then
\[
\g_1\cong\mathfrak{o}_5(\C)\oplus\C H_1, [\mathfrak{o}_5(\C),H_1]=0;
\]
and
\[
\g_3\cong\mathfrak{sl}_3(\C)\oplus\C H_3, [\mathfrak{sl}_3(\C),H_3]=0.
\]
Let $M_k=\mathrm{Res}_{\g_k}^{\g}M$ and $\h_k$ be the subalgebra of $\h$ generated by $\{\alpha_i^\vee|i\neq k\}$. For $i=1,2$ take $\bar{H}_i\in\h_3$ such that
\[
[\bar{H}_i,e_j]=\delta_{ij}e_j, [\bar{H}_i,f_j]=-\delta_{ij}f_j, \,\,\, i,j=1,2.
\]
From $$\aligned
2\bar H_1-\bar H_2&=\alpha_1^\vee=2 H_1- H_2,\\
-\bar H_1+2\bar H_2&=\alpha_2^\vee=-H_1+2H_2-H_3, \endaligned$$ we see that
\[
\bar{H}_i=H_i-\frac{i}{3}H_3,i=1,2.
\]
Considering $M_3$,  from Lemma \ref{AXLEM}, we see that $(\S_2,\P_2)$ is one of the following:
\begin{equation}\label{R5.2}
\left\{\begin{array}{l}
(a(H_2-\frac{2}{3}H_3+b),\,\,\,\,-a^{-1}(H_2-H_1-\frac{1}{3}H_3-b)),\\
(a,\,\,\,\,-a^{-1}(H_2-H_1-\frac{1}{3}H_3-b)(H_2-\frac{2}{3}H_3+b+1)),\\
(-a(H_2-H_1-\frac{1}{3}H_3-b-1)(H_2-\frac{2}{3}H_3+b),\,\,\,\,a^{-1}),\\
(a(H_2-H_1-\frac{1}{3}H_3-b-1),\,\,\,\,-a^{-1}(H_2-\frac{2}{3}H_3+b+1)),
\end{array}\right.
\end{equation}
where $a\in\C^*,b\in\C[H_3]$.

Similarly, considering $M_1$, we know that $(\S_2,\P_2)$ is one of the following
\begin{equation}\label{R5.3}
\left\{\begin{array}{l}
(a,\,\,-a^{-1}(H_2-\frac{1}{2}H_3-\frac{1}{2}H_1-\frac{3}{4})(H_2-\frac{1}{2}H_3-\frac{1}{2}H_1-\frac{1}{4})),\\
(-a(H_2-\frac{1}{2}H_3-\frac{1}{2}H_1+\frac{3}{4})(H_2-\frac{1}{2}H_3-\frac{1}{2}H_1+\frac{1}{4})£,\,\,a^{-1}),
\end{array}\right.
\end{equation}
where $a\in\C^*$.
Compare (\ref{R5.2}) with (\ref{R5.3}), we have
\[
(\deg_1\S_2,\deg_1\P_2)=(0,1)=(0,2),
\]
or
\[
(\deg_1\S_2,\deg_1\P_2)=(1,0)=(2,0).
\]
This is absurd. Hence, $M$ does not exist, that is $\H(\g)=\varnothing$.
\end{proof}

Thus, for Kac-Moody algebras of finite type, we only have  type $C_l \,\, (l>2)$ left. The complete list of modules in $\H(C_l)$ is as follows, which were not explicitly given by Nilsson in \cite{N1}

\begin{proposition}\label{C}
Let $\g$ be a Kac-Moody algebra of type $C_l\,\,(l\geq 2)$. For any $\a=(a_1,\cdots,a_l)\in(\C^*)^l, S\subseteq\l$, the space $M(\a,S)=\C[H_1,\cdots,H_l]$ becomes a $\g$-module under the following action for all $g\in\C[H_1,\cdots,H_l], 1\leq k\leq l-2$:
\[\begin{aligned}
H_i\cdot g&=H_ig, i\in\l;\\
e_k\cdot g&=\hskip-0.1cm\left\{\begin{array}{ll}
\hskip-0.3cm a_k(H_{k+1}-H_k+\frac{1}{2})\s_k(g), &\hskip-0.2cm k,k+1\in S,\\
\hskip-0.3cm a_k\s_k(g), &\hskip-0.2cm k\in S,k+1\not\in S,\\
\hskip-0.3cm a_k(H_k-H_{k-1}-\frac{1}{2})(H_{k+1}-H_k+\frac{1}{2})\s_k(g), &\hskip-0.2cm k\not\in S,k+1\in S,\\
\hskip-0.3cm a_k(H_k-H_{k-1}-\frac{1}{2})\s_k(g), &\hskip-0.2cm k,k+1\not\in S;
\end{array}\right.\end{aligned}\]
\[\begin{aligned}
f_k\cdot g&=\hskip-0.1cm\left\{\begin{array}{ll}
\hskip-0.3cm a_k^{-1}(H_k-H_{k-1}+\frac{1}{2})\s_k^{-1}(g), &\hskip-0.2cm k,k+1\in S,\\
\hskip-0.3cm a_k^{-1}(H_k-H_{k-1}+\frac{1}{2})(H_{k+1}-H_k-\frac{1}{2})\s_k^{-1}(g), &\hskip-0.2cm k\in S,k+1\not\in S,\\
\hskip-0.3cm a_k^{-1}\s_k^{-1}(g), &\hskip-0.2cm k\not\in S,k+1\in S,\\
\hskip-0.3cm a_k^{-1}(H_{k+1}-H_k-\frac{1}{2})\s_k^{-1}(g), &\hskip-0.2cm k,k+1\not\in S;
\end{array}\right.\\
e_{l-1}\cdot g&=\hskip-0.1cm\left\{\begin{array}{ll}
\hskip-0.3cm a_{l-1}(2H_l-H_{l-1}+\frac{1}{2})\s_{l-1}(g), &\hskip-0.2cm l-1,l\in S,\\
\hskip-0.3cm a_{l-1}\s_{l-1}(g), &\hskip-0.2cm l-1\in S, l\not\in S,\\
\hskip-0.3cm a_{l-1}(H_{l-1}-H_{l-2}-\frac{1}{2})(2H_{l}-H_{l-1}+\frac{1}{2})\s_{l-1}(g), &\hskip-0.2cm l-1\not\in S, l\in S,\\
\hskip-0.3cm a_{l-1}(H_{l-1}-H_{l-2}-\frac{1}{2})\s_{l-1}(g), &\hskip-0.2cm l-1,l\not\in S;
\end{array}\right.\\
f_{l-1}\cdot g&=\hskip-0.1cm\left\{\begin{array}{ll}
\hskip-0.3cm a_{l-1}^{-1}(H_{l-1}-H_{l-2}+\frac{1}{2})\s_{l-1}^{-1}(g), &\hskip-0.2cm l-1,l\in S,\\
\hskip-0.3cm a_{l-1}^{-1}(H_{l-1}-H_{l-2}+\frac{1}{2})(2H_{l}-H_{l-1}-\frac{1}{2})\s_{l-1}^{-1}(g), &\hskip-0.2cm l-1\in S, l\not\in S,\\
\hskip-0.3cm a_{l-1}^{-1}\s_{l-1}^{-1}(g), &\hskip-0.2cm l-1\not\in S, l\in S,\\
\hskip-0.3cm a_{l-1}^{-1}(2H_l-H_{l-1}-\frac{1}{2})\s_{l-1}^{-1}(g), &\hskip-0.2cm l-1,l\not\in S;
\end{array}\right.\\
e_l\cdot g&=\hskip-0.1cm\left\{\begin{array}{ll}
\hskip-0.3cm a_l\s_l(g), &\hskip-0.2cm l\in S,\\
\hskip-0.3cm-a_l(H_l-\frac{1}{2}H_{l-1}-\frac{3}{4})(H_l-\frac{1}{2}H_{l-1}-\frac{1}{4})\s_l(g), &\hskip-0.2cm l\not\in S;
\end{array}\right.\\
f_l\cdot g&=\hskip-0.1cm\left\{\begin{array}{ll}
\hskip-0.3cm-a_l^{-1}(H_l-\frac{1}{2}H_{l-1}+\frac{3}{4})(H_l-\frac{1}{2}H_{l-1}+\frac{1}{4})\s_l^{-1}(g), &\hskip-0.2cm l\in S,\\
\hskip-0.3cm a_l^{-1}\s_l^{-1}(g), &\hskip-0.2cm l\not\in S.
\end{array}\right.
\end{aligned}\]
where $H_0:=0$. And we have $\H(\g)=\{M(\a,S)|\a\in(\C^*)^l,S\subseteq\l\}.$
Moreover, the module  $M(\a,S)$ is simple for any $\a$ and $S$.
\end{proposition}

\begin{proof}
Suppose the Dynkin diagram of $\g$ is as follows.
\[
\begin{tikzpicture}[scale=0.7]
    \draw (1,0) -- (2,0);
    \draw (2,0) -- (3,0);
    \draw (4,0) -- (5,0);
    \draw (5,0.1) -- (6,0.1);
    \draw (5,-0.1) -- (6,-0.1);
    \draw (5.6,0.2) -- (5.4,0) -- (5.6,-0.2);
    \filldraw[color=black] (1,0) circle (5pt) node[below=2pt]{\small 1};
    \filldraw[color=white] (1,0) circle (4pt);
    \filldraw[color=black] (2,0) circle (5pt) node[below=2pt]{\small 2};
    \filldraw[color=white] (2,0) circle (4pt);
    \filldraw[color=black] (5,0) circle (5pt) node[below=2pt]{\small $l-1$};
    \filldraw[color=white] (5,0) circle (4pt);
    \filldraw[color=black] (6,0) circle (5pt) node[below=2pt]{\small $l$};
    \filldraw[color=white] (6,0) circle (4pt);
    \filldraw (3,0) node[right]{$\cdots$};
\end{tikzpicture}
\]
The statement for $l=2$ is just Theorem \ref{RANK2}(2). Now we assume that $l\geq3$ and $M\in\H(\g)$. Let $\g_1$ be the subalgebra of $\g$ generated by $\{e_i,f_i,\h|i=1,\cdots,l-1\}$  and $\g_2$ be the subalgebra of $\g$ generated by $\{e_{l-1},e_l,f_{l-1},f_l,\h\}$. Then
\[\begin{aligned}
&\g_1\cong\mathfrak{sl}_l(\C)\oplus\C H_l, [\sl_{l}(\C),H_l]=0;\\
&\g_2\cong\mathfrak{o}_5(\C)\oplus\C H_1\cdots\oplus\C H_{l-2}, [\mathfrak{o}_5(\C),H_i]=0, i=1,\cdots,l-2.
\end{aligned}\]
Let $M_i=\mathrm{Res}_{\g_i}^{\g}M, i=1,2$. Let $\h_1$ be the subalgebra of $\h$ generated by $\{\alpha_i^\vee|i=1,\cdots, l-1\}$ and $\h_2$ be the subalgebra of $\h$ generated by $\{\alpha_{l-1}^\vee,\alpha_l^\vee\}$.

Take $\bar{H}_i$ in $\h_1$ such that
\[
[\bar{H}_i,e_j]=\delta_{ij}e_j, [\bar{H}_i,f_j]=-\delta_{ij}f_j, 1\leq i,j\leq l-1.
\]
Then
\[
\bar{H}_i=H_i-\frac{2i}{l}H_l.
\]
Following from Lemma \ref{AXLEM}(1), on $M_1$, we have
\[
(\S_i,\P_i)={\tiny \left\{\begin{array}{ll}
(a_i(H_{i+1}-H_i-\frac{2}{l}H_l-b),a_i^{-1}(H_i-H_{i-1}-\frac{2}{l}H_l-b)), & i,i+1\in S,\\
(a_i,a_i^{-1}(H_i-H_{i-1}-\frac{2}{l}H_l-b)(H_{i+1}-H_i-\frac{2}{l}H_l-b-1)), & i\in S,i+1\not\in S,\\
(a_i(H_i-H_{i-1}-\frac{2}{l}H_l-b-1)(H_{i+1}-H_i-\frac{2}{l}H_l-b),a_i^{-1}), & i\not\in S,i+1\in S,\\
(a_i(H_i-H_{i-1}-\frac{2}{l}H_l-b-1),a_i^{-1}(H_{i+1}-H_i-\frac{2}{l}H_l-b), & i,i+1\not\in S,
\end{array}\right.}
\]
for $1\leq i\leq l-2$, and
\begin{equation}\label{R5.4}\begin{aligned}
&(\S_{l-1},\P_{l-1})=\\
&{\tiny \left\{\begin{array}{ll}
(a_{l-1}(\frac{2(l-1)}{l}H_l-H_{l-1}-b),a_{l-1}^{-1}(H_{l-1}-H_{l-2}-\frac{2}{l}H_l-b)), & l-1,l\in S,\\
(a_{l-1},a_{l-1}^{-1}(H_{l-1}-H_{l-2}-\frac{2}{l}H_l-b)(\frac{2(l-1)}{l}H_l-H_{l-1}-b-1)), & l-1\in S, l\not\in S,\\
(a_{l-1}(H_{l-1}-H_{l-2}-\frac{2}{l}H_l-b-1)(\frac{2(l-1)}{l}H_l-H_{l-1}-b),a_{l-1}^{-1}), & l-1\not\in S, l\in S,\\
(a_{l-1}(H_{l-1}-H_{l-2}-\frac{2}{l}H_l-b-1),a_{l-1}^{-1}(\frac{2(l-1)}{l}H_l-H_{l-1}-b-1)), & l-1,l\not\in S.
\end{array}\right.}
\end{aligned}\end{equation}

Similarly, on $M_2$, by Lemma \ref{AXLEM}(2), we have
\begin{equation}\label{R5.5}\begin{aligned}
(\S_{l-1},\P_{l-1})&={\tiny\left\{\begin{array}{ll}
(a_{l-1}(2H_l-H_{l-1}+\frac{1}{2}),a_{l-1}^{-1}(H_{l-1}-H_{l-2}+\frac{1}{2})), & l-1,l\in S,\\
(a_{l-1},a_{l-1}^{-1}(H_{l-1}-H_{l-2}+\frac{1}{2})(2H_l-H_{l-1}-\frac{1}{2})), & l-1\in S, l\not\in S,\\
(a_{l-1}(H_{l-1}-H_{l-2}-\frac{1}{2})(2H_l-H_{l-1}+\frac{1}{2}),a_{l-1}^{-1}), & l-1\not\in S, l\in S,\\
(a_{l-1}(H_{l-1}-H_{l-2}-\frac{1}{2}),a_{l-1}^{-1}(2H_l-H_{l-1}-\frac{1}{2})), & l-1,l\not\in S;
\end{array}\right.}\\
(\S_l,\P_l)&=\left\{\begin{array}{ll}
(a_l,-a_l^{-1}(H_l-\frac{1}{2}H_{l-1}+\frac{3}{4})(H_l-\frac{1}{2}H_{l-1}+\frac{1}{4})), & l\in S,\\
(-a_l(H_l-\frac{1}{2}H_{l-1}-\frac{3}{4})(H_l-\frac{1}{2}H_{l-1}-\frac{1}{4}),a_l^{-1}), & l\not\in S.
\end{array}\right.
\end{aligned}\end{equation}
Compare (\ref{R5.4}) with (\ref{R5.5}),  we get
\[
b=-\frac{2}{l}H_l-\frac{1}{2},
\]
and hence $M\cong M(\a,S)$ for some $\a\in(\C^*)^l$ and $S\subseteq\l$.

For irreducibility, we only show that $M(\mathbf{1},\l)$ is simple for the others can be proved in similar ways.

Define the following operators on $M(\mathbf{1},\l)$:
\[
\Delta_i:=e_i-1, i\in\l.
\]
Let $W\subseteq M(\mathbf{1},\l)$ be a nonzero submodule and take $0\neq g\in W$. Applying $\Delta_l$ finite times, we have
\[
g_1\in W\cap(\cP_l\setminus\{0\}).
\]
Then applying $\Delta_l\circ e_{l-1}$, we get
\[
\s_{l-1}(g_1)\in W.
\]
Hence
\[
\s_{l-1}(g_1)-g_1\in W.
\]
Thus, we have
\[
0\neq g_2\in W\cap\cP_l\cap\cP_{l-1}.
\]
Continue this process by using $\Delta_i\circ e_{i-1}$ in turn, we know that there exists $g'\in\C^*\cap W$. Therefore, $1\in W$ and hence $W=M(\mathbf{1},\l)$.
\end{proof}

Following from Example 1, Corollary \ref{DE}, Lemma \ref{B,F} and Proposition \ref{C}, we have
\begin{theorem}\label{FINITE}
Let $\g$ be a Kac-Moody algebra of finite type. Then $\H(\g)\neq\varnothing$ if and only if $\g$ is of type $A_l\,\,(l\geq1)$ or $C_l\,\,(l\geq2)$.
\end{theorem}

\section{Modules for general Kac-Moody algebras}

In this section, we determine the category $\H(\g)$ for general Kac-Moody algebras $\g$. Indeed, we will prove
\begin{theorem}
Let $\g$ be an arbitrary  Kac-Moody algebra. Then $\H(\g)\neq\varnothing$ if and only if $\g$ is of finite type $A_l\,\,(l\geq1)$ or $C_l\,\,(l\geq2)$.
\end{theorem}

Let $\g$ be a Kac-Moody algebra associtaed to Cartan matrix $A=(a_{ij})$ with Dynkin diagram $\Gamma$ and $\H(\g)\neq\varnothing$. Then using Theorem \ref{AFF}, Theorem \ref{RANK2}, Lemma \ref{AXLEM} and Theorem \ref{FINITE}, $\Gamma$ has the following properties:
\begin{itemize}
\item $\Gamma$ cannot contain a triple line, a quadruple line, a bold line or a double arrow line, i.e., $a_{ij}a_{ji}\le2$. See Theorem \ref{RANK2}.
\item $\Gamma$ cannot contain a subdiagram of type $D_4,B_3^{(1)},A_3^{(2)},C_2^{(1)},A_4^{(2)}$ and $D_3^{(2)}$. See Theorem \ref{AFF}.
\item If $\Gamma$ is simple-laced, then $\Gamma$ has to be of type $A_l$: this follows from the fact that $\H=\varnothing$ for Kac-Moody algebras of type $A_l^{(1)}$.
\item If $\Gamma$ has more than 3 points without a 3-point circle and contains a double line, then $\Gamma$ has to be of type $C_l(l\geq4)$: otherwise, $\Gamma$ will contain a subdiagram of type $B_3$ or affine type which has $\H=\varnothing$.
\end{itemize}

To complete the proof of Theorem 20, we only need to prove the following lemma which consists of all possible 3-point circle Dynkin diagrams with at least one double line.
\begin{lemma}
Let $\g$ be a Kac-Moody algebra with Dynkin diagram $\Gamma$. If $\Gamma$ is one of the following,
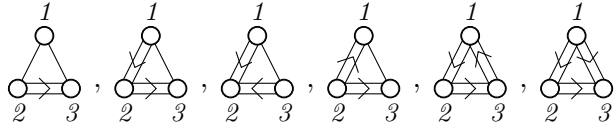
\[
\begin{tikzpicture}[scale=0.7]
    \draw (1,0.1) -- (2,0.1);
    \draw (1,-0.1) -- (2,-0.1);
    \draw (1.4,0.2) -- (1.6,0) -- (1.4,-0.2);
    \draw (1,0)--(1.5,1);
    \draw (2,0)--(1.5,1);
    \filldraw[color=black] (1,0) circle (5pt) node[below=2pt]{\small 2};
    \filldraw[color=white] (1,0) circle (4pt);
    \filldraw[color=black] (2,0) circle (5pt) node[below=2pt]{\small 3};
    \filldraw[color=white] (2,0) circle (4pt);
    \filldraw[color=black] (1.5,1) circle (5pt) node[above=2pt]{\small 1};
    \filldraw[color=white] (1.5,1) circle (4pt);
    \filldraw (2.5,0) node{$,$};
    \draw (3,0.1) -- (4,0.1);
    \draw (3,-0.1) -- (4,-0.1);
    \draw (3.4,0.2) -- (3.6,0) -- (3.4,-0.2);
    \draw (2.9,0) -- (3.4,1);
    \draw (3.1,0) -- (3.6,1);
    \draw (3.1,0.7) -- (3.2,0.4) -- (3.4,0.5);
    \draw (4,0) -- (3.5,1);
    \filldraw[color=black] (3,0) circle (5pt) node[below=2pt]{\small 2};
    \filldraw[color=white] (3,0) circle (4pt);
    \filldraw[color=black] (4,0) circle (5pt) node[below=2pt]{\small 3};
    \filldraw[color=white] (4,0) circle (4pt);
    \filldraw[color=black] (3.5,1) circle (5pt) node[above=2pt]{\small 1};
    \filldraw[color=white] (3.5,1) circle (4pt);
    \filldraw (4.5,0) node{$,$};
    \draw (5,0.1) -- (6,0.1);
    \draw (5,-0.1) -- (6,-0.1);
    \draw (5.6,0.2) -- (5.4,0) -- (5.6,-0.2);
    \draw (4.9,0) -- (5.4,1);
    \draw (5.1,0) -- (5.6,1);
    \draw (5.1,0.7) -- (5.2,0.4) -- (5.4,0.5);
    \draw (6,0) -- (5.5,1);
    \filldraw[color=black] (5,0) circle (5pt) node[below=2pt]{\small 2};
    \filldraw[color=white] (5,0) circle (4pt);
    \filldraw[color=black] (6,0) circle (5pt) node[below=2pt]{\small 3};
    \filldraw[color=white] (6,0) circle (4pt);
    \filldraw[color=black] (5.5,1) circle (5pt) node[above=2pt]{\small 1};
    \filldraw[color=white] (5.5,1) circle (4pt);
    \filldraw (6.5,0) node{$,$};
    \draw (7,0.1) -- (8,0.1);
    \draw (7,-0.1) -- (8,-0.1);
    \draw (7.4,0.2) -- (7.6,0) -- (7.4,-0.2);
    \draw (6.9,0) -- (7.4,1);
    \draw (7.1,0) -- (7.6,1);
    \draw (7,0.5) -- (7.3,0.6) -- (7.4,0.3);
    \draw (8,0) -- (7.5,1);
    \filldraw[color=black] (7,0) circle (5pt) node[below=2pt]{\small 2};
    \filldraw[color=white] (7,0) circle (4pt);
    \filldraw[color=black] (8,0) circle (5pt) node[below=2pt]{\small 3};
    \filldraw[color=white] (8,0) circle (4pt);
    \filldraw[color=black] (7.5,1) circle (5pt) node[above=2pt]{\small 1};
    \filldraw[color=white] (7.5,1) circle (4pt);
    \filldraw (8.5,0) node{$,$};
    \draw (9,0.1) -- (10,0.1);
    \draw (9,-0.1) -- (10,-0.1);
    \draw (9.4,0.2) -- (9.6,0) -- (9.4,-0.2);
    \draw (8.9,0) -- (9.4,1);
    \draw (9.1,0) -- (9.6,1);
    \draw (9.1,0.7) -- (9.2,0.4) -- (9.4,0.5);
    \draw (9.9,0) -- (9.4,1);
    \draw (10.1,0) -- (9.6,1);
    \draw (9.6,0.4) -- (9.65,0.7) -- (9.9, 0.6);
    \filldraw[color=black] (9,0) circle (5pt) node[below=2pt]{\small 2};
    \filldraw[color=white] (9,0) circle (4pt);
    \filldraw[color=black] (10,0) circle (5pt) node[below=2pt]{\small 3};
    \filldraw[color=white] (10,0) circle (4pt);
    \filldraw[color=black] (9.5,1) circle (5pt) node[above=2pt]{\small 1};
    \filldraw[color=white] (9.5,1) circle (4pt);
    \filldraw (10.5,0) node{$,$};
    \draw (11,0.1) -- (12,0.1);
    \draw (11,-0.1) -- (12,-0.1);
    \draw (11.4,0.2) -- (11.6,0) -- (11.4,-0.2);
    \draw (10.9,0) -- (11.4,1);
    \draw (11.1,0) -- (11.6,1);
    \draw (11.1,0.7) -- (11.2,0.4) -- (11.4,0.5);
    \draw (11.9,0) -- (11.4,1);
    \draw (12.1,0) -- (11.6,1);
    \draw (11.5,0.6) -- (11.75,0.5) -- (11.9, 0.7);
    \filldraw[color=black] (11,0) circle (5pt) node[below=2pt]{\small 2};
    \filldraw[color=white] (11,0) circle (4pt);
    \filldraw[color=black] (12,0) circle (5pt) node[below=2pt]{\small 3};
    \filldraw[color=white] (12,0) circle (4pt);
    \filldraw[color=black] (11.5,1) circle (5pt) node[above=2pt]{\small 1};
    \filldraw[color=white] (11.5,1) circle (4pt);
\end{tikzpicture}
\]
then $\H(\g)=\varnothing$.
\end{lemma}
\begin{proof}
The Cartan matrix $A$ for $\g$ is one of the following, respectively,
\[\begin{aligned}
&\begin{pmatrix}
2 & -1 & -1 \\
-1 & 2 & -1 \\
-1  & -2 & 2
\end{pmatrix},\begin{pmatrix}
2 & -1 & -1 \\
-2 & 2 & -1 \\
-1  & -2 & 2
\end{pmatrix},\begin{pmatrix}
2 & -1 & -1 \\
-2 & 2 & -2 \\
-1  & -1 & 2
\end{pmatrix},\\
&\begin{pmatrix}
2 & -2 & -1 \\
-1 & 2 & -1 \\
-1  & -2 & 2
\end{pmatrix},\begin{pmatrix}
2 & -1 & -2 \\
-2 & 2 & -1 \\
-1  & -2 & 2
\end{pmatrix},\begin{pmatrix}
2 & -1 & -1 \\
-2 & 2 & -1 \\
-2  & -2 & 2
\end{pmatrix}
\end{aligned}\]
Clearly, $A$ is invertible. Let $\{H_1,H_2,H_3\}$ be the dual basis for $\alpha_1,\alpha_2,\alpha_3$.
For $k=1,2,3$, let $\g_k$ be the subalgebra generated by $\{e_i,f_i,\h|i\in\{1,2,3\}\setminus\{k\}\}$.
Suppose that $\H(\g)\neq\varnothing$ and let $M\in\H(\g)$. Define $M_k=\mathrm{Res}_{\g_k}^{\g}M$ for $k=1,2,3$, and $\S_i=e_i\cdot1,\P_i=f_i\cdot1$.
We will cary out the proof case by case.

{\bf Case 1:}
$A=\begin{pmatrix}
2 & -1 & -1 \\
-1 & 2 & -1 \\
-1  & -2 & 2
\end{pmatrix}$.

In $M_3$, from Lemma \ref{A_1}, we know that $(\S_2,\P_2)$ is one of the following
\[
\left\{\begin{array}{l}
(a(H_2-H_3+b),a^{-1}(H_1-H_2+b)),\\
(a,-a^{-1}(H_2-H_1-b)(H_2-H_3+b+1)),\\
(-a(H_2-H_1-b-1)(H_2-H_3+b),a^{-1}),\\
(a(H_2-H_1-b-1),a^{-1}(H_3-H_2-b-1)),
\end{array}\right.
\]
where $a\in\C^*,b\in\C[H_3]$. Therefore,
\[
(\deg_1\S_2,\deg_1\P_2)=(0,1) \text{ or } (1,0).
\]
However, in $M_1$, following from Lemma \ref{AXLEM}, there exists $c\in\C^*$ such that $(\S_2,\P_2)$ is one of the following
\[
\left\{\begin{array}{l}
(c,-c^{-1}(H_2-\frac{1}{2}H_1-\frac{1}{2}H_3-\frac{3}{4})(H_2-\frac{1}{2}H_1-\frac{1}{2}H_3-\frac{1}{4})),\\
(-c(H_2-\frac{1}{2}H_1-\frac{1}{2}H_3+\frac{3}{4})(H_2-\frac{1}{2}H_1-\frac{1}{2}H_3+\frac{1}{4}),c^{-1}).
\end{array}\right.
\]
So,
\[
(\deg_1\S_2,\deg_1\P_2)=(0,2) \text{ or } (2,0).
\]
This leads to a contradiction. Hence, $\H(\g)=\varnothing$.

{\bf Case 2:} $A=\begin{pmatrix}
2 & -1 & -1 \\
-2 & 2 & -1 \\
-1  & -2 & 2
\end{pmatrix}$.

In $M_3$, there exists $a\in\C^*$ such that $(\S_1,\P_1)$ is one of the following
\[
\left\{\begin{array}{l}
(a,-a^{-1}(H_1-\frac{1}{2}H_2-\frac{1}{2}H_3-\frac{3}{4})(H_1-\frac{1}{2}H_2-\frac{1}{2}H_3-\frac{1}{4})),\\
(-a(H_1-\frac{1}{2}H_2-\frac{1}{2}H_3+\frac{3}{4})(H_1-\frac{1}{2}H_2-\frac{1}{2}H_3+\frac{1}{4}),a^{-1}).
\end{array}\right.
\]
Hence,
\[
(\deg_3\S_1,\deg_3\P_1)=(0,2) \text{ or } (2,0).
\]
In $M_2$, there exists $c\in\C^*,b\in\C[H_2]$ such that $(\S_1,\P_1)$ is one of the following
\[
\left\{\begin{array}{l}
(c(H_3-H_1-\frac{1}{3}H_2-b),c^{-1}(H_1-\frac{4}{3}H_2-b)),\\
(c,c^{-1}(H_1-\frac{4}{3}H_2-b)(H_3-H_1-\frac{1}{3}H_2-b-1)),\\
(c(H_1-\frac{4}{3}H_2-b-1)(H_3-H_1-\frac{1}{3}H_2-b),c^{-1}),\\
(c(H_1-\frac{4}{3}H_2-b-1),c^{-1}(H_3-H_1-\frac{1}{3}H_2-b-1)).
\end{array}\right.
\]
Therefore,
\[
(\deg_3\S_1,\deg_3\P_1)=(0,1) \text{ or } (1,0).
\]
This is a contradiction. So, $\H(\g)=\varnothing$.

{\bf Case 3:} $A=\begin{pmatrix}
2 & -1 & -1 \\
-2 & 2 & -2 \\
-1  & -1 & 2
\end{pmatrix}$.

In $M_3$, there exists $a\in\C^*$ such that $(\S_1,\P_1)$ is one of the following
\[
\left\{\begin{array}{l}
(a,-a^{-1}(H_1-\frac{1}{2}H_2-\frac{1}{2}H_3-\frac{3}{4})(H_1-\frac{1}{2}H_2-\frac{1}{2}H_3-\frac{1}{4})),\\
(-a(H_1-\frac{1}{2}H_2-\frac{1}{2}H_3+\frac{3}{4})(H_1-\frac{1}{2}H_2-\frac{1}{2}H_3+\frac{1}{4}),a^{-1}).
\end{array}\right.
\]
Hence,
\[
(\deg_3\S_1,\deg_3\P_1)=(0,2) \text{ or } (2,0).
\]
In $M_2$, there exists $c\in\C^*,b\in\C[H_2]$ such that $(\S_1,\P_1)$ is one of the following
\[
\left\{\begin{array}{l}
(c(H_3-H_1-b),c^{-1}(H_1-H_2-b)),\\
(c,c^{-1}(H_1-H_2-b)(H_3-H_1-b-1)),\\
(c(H_1-H_2-b-1)(H_3-H_1H_2-b),c^{-1}),\\
(c(H_1-H_2-b-1),c^{-1}(H_3-H_1-b-1)).
\end{array}\right.
\]
Therefore,
\[
(\deg_3\S_1,\deg_3\P_1)=(0,1) \text{ or } (1,0).
\]
This is a contradiction. So, $\H(\g)=\varnothing$.

{\bf Case 4:} $A=\begin{pmatrix}
2 & -2 & -1 \\
-1 & 2 & -1 \\
-1  & -2 & 2
\end{pmatrix}$.

In $M_3$, there exists $a\in\C^*$ such that $(\S_1,\P_1)$ is one of the following
\[
\left\{\begin{array}{l}
(a(2H_2-H_1-H_3+\frac{1}{2}),a^{-1}(H_1-2H_3+\frac{1}{2})),\\
(a,-a^{-1}(H_1-2H_3+\frac{1}{2})(2H_2-H_1-H_3-\frac{1}{2})),\\
(a(H_1-2H_3-\frac{1}{2})(2H_2-H_1-H_3+\frac{1}{2}),a^{-1}),\\
(a(H_1-2H_3-\frac{1}{2}),a^{-1}(2H_2-H_1-H_3-\frac{1}{2})).
\end{array}\right.
\]
Hence,
\[
(\deg_3\S_1,\deg_3\P_1)=(0,2),(2,0) \text{ or } (1,1).
\]
In $M_2$, there exists $c\in\C^*,b\in\C[H_2]$ such that $(\S_1,\P_1)$ is one of the following
\[
\left\{\begin{array}{l}
(c(H_3-H_1-b),c^{-1}(H_1-2H_2-b)),\\
(c,c^{-1}(H_1-2H_2-b)(H_3-H_1-b-1)),\\
(c(H_1-2H_2-b-1)(H_3-H_1-b),c^{-1}),\\
(c(H_1-2H_2-b-1),c^{-1}(H_3-H_1-b-1)).
\end{array}\right.
\]
Therefore,
\[
(\deg_3\S_1,\deg_3\P_1)=(0,1) \text{ or } (1,0).
\]
This is a contradiction. So, $\H(\g)=\varnothing$.

{\bf Case 5:} $A=\begin{pmatrix}
2 & -1 & -2 \\
-2 & 2 & -1 \\
-1  & -2 & 2
\end{pmatrix}$.

In $M_3$, there exists $a\in\C^*$ such that $(\S_1,\P_1)$ is one of the following
\[
\left\{\begin{array}{l}
(a,-a^{-1}(H_1-\frac{1}{2}H_2-H_3+\frac{3}{4})(H_1-\frac{1}{2}H_2-H_3+\frac{1}{4})),\\
(-a(H_1-\frac{1}{2}H_2-H_3-\frac{3}{4})(H_1-\frac{1}{2}H_2-H_3-\frac{1}{4}),a^{-1}).
\end{array}\right.
\]
Hence,
\[
(\deg_3\S_1,\deg_3\P_1)=(0,2) \text{ or } (2,0).
\]
In $M_2$, there exists $c\in\C^*$ such that $(\S_1,\P_1)$ is one of the following
\[
\left\{\begin{array}{l}
(c(2H_3-H_1-{2}{H_2}+\frac{1}{2}),c^{-1}(H_1-{3}H_2+\frac{1}{2})),\\
(c,c^{-1}(H_1-{3}H_2+\frac{1}{2})(2H_3-H_1-{2}{H_2}-\frac{1}{2})),\\
(c(H_1-{3}H_2-\frac{1}{2})(2H_3-H_1-{2}{H_2}+\frac{1}{2}),c^{-1}),\\
(c(H_1-{3}H_2-\frac{1}{2}),c^{-1}(2H_3-H_1-{2}{H_2}-\frac{1}{2}))).
\end{array}\right.
\]
Therefore,
\[
(\deg_3\S_1,\deg_3\P_1)=(0,1) \text{ or } (1,0).
\]
This is a contradiction. So, $\H(\g)=\varnothing$.

{\bf Case 6:} $A=\begin{pmatrix}
2 & -1 & -1 \\
-2 & 2 & -1 \\
-2  & -2 & 2
\end{pmatrix}$.

In $M_1$, there exists $a\in\C^*$ such that $(\S_2,\P_2)$ is one of the following
\[
\left\{\begin{array}{l}
(a,-a^{-1}(H_2-\frac{1}{2}H_3-H_1-\frac{3}{4})(H_2-\frac{1}{2}H_3-H_1-\frac{1}{4})),\\
(-a(H_2-\frac{1}{2}H_3-H_1+\frac{3}{4})(H_2-\frac{1}{2}H_3-H_1+\frac{1}{4}),a^{-1}).
\end{array}\right.
\]
Hence,
\[
(\deg_1\S_2,\deg_1\P_2)=(0,2) \text{ or } (2,0).
\]
In $M_3$, there exists $c\in\C^*$ such that $(\S_2,\P_2)$ is one of the following
\[
\left\{\begin{array}{l}
(c(2H_1-H_2-H_3+\frac{1}{2}),c^{-1}(H_2-2H_3+\frac{1}{2})),\\
(c,c^{-1}(H_2-2H_3+\frac{1}{2})(2H_1-H_2-H_3-\frac{1}{2})),\\
(c(H_2-2H_3-\frac{1}{2})(2H_1-H_2-H_3+\frac{1}{2}),c^{-1}),\\
(c(H_2-2H_3-\frac{1}{2}),c^{-1}(2H_1-H_2-H_3-\frac{1}{2}))).
\end{array}\right.
\]
Therefore,
\[
(\deg_1\S_2,\deg_1\P_2)=(0,1) \text{ or } (1,0).
\]
This is a contradiction. So, $\H(\g)=\varnothing$.
\end{proof}

\

\noindent {\bf Acknowledgments.} The research presented in this
paper was carried out during the visit of Y.C. and H.T. to Wilfrid
Laurier University supported by University Research Professor grant
in the winter of 2014/2015.  K.Z. is partially
supported by NSF of China (Grant 11271109, 11371134) and NSERC (Grant 311907-2015).
  Y.C. likes to thank the China Scholar Council for the financial support.

\vspace{10mm}

\noindent Y. Cai: Academy of Mathematics and Systems Science, Chinese Academy of Sciences, Beijing, 100190, P.R. China, and Wu Wen Tsun Key Laboratory of Mathematics, Chinese Academy of Science, School of Mathematical Sciences, University of Science and Technology of China, Hefei, 230026, Anhui, P. R. China.
Email: yatsai@mail.ustc.edu.cn

\vspace{0.2cm}
\noindent H. Tan: School of Mathematics and Statistics, Northeast Normal University, Changchun, Jilin, 130042, P.R. China, and Department of Applied Mathematics, Changchun University of Science and Technology, Changchun, Jilin,
130022, P.R. China.
Email: hjtan9999@yahoo.com

\vspace{0.2cm}
 \noindent K. Zhao: College of
Mathematics and Information Science, Hebei Normal (Teachers)
University, Shijiazhuang, Hebei, 050016 P. R. China, and Department of Mathematics, Wilfrid
Laurier University, Waterloo, ON, Canada N2L 3C5. Email:
kzhao@wlu.ca

\end{document}